\documentclass[12pt]{amsart}

\usepackage{amsfonts,amsmath,amsthm,amssymb,latexsym,amsthm,newlfont,enumerate,color,tikz-cd}

\newif\ifslide

\theoremstyle{plain}
\newtheorem{theorem}{Theorem}[section]

\newtheorem{corollary}[theorem]{Corollary}
\newtheorem{lemma}[theorem]{Lemma}
\newtheorem{claim}[theorem]{Claim}

\newtheorem{proposition}[theorem]{Proposition}
\newtheorem{definition-lemma}[theorem]{Definition-Lemma}

\newtheorem{red-question}[theorem]{\textcolor{red}{Question}}

\theoremstyle{definition}
\newtheorem{definition}[theorem]{Definition}
\newtheorem{remark}[theorem]{Remark}

\def\ideal#1.{I_{#1}}
\def\ring#1.{\mathcal {O}_{#1}}

\def\fring#1.{\hat{\mathcal {O}}_{#1}}
\def\proj#1.{\mathbb {P}(#1)}
\def\pr #1.{\mathbb {P}^{#1}}
\def\dpr #1.{\hat{\mathbb {P}}^{#1}}
\def\af #1.{\mathbb A^{#1}}
\def\Hz #1.{\mathbb F_{#1}}
\def\Hbz #1.{\overline{\mathbb F}_{#1}}
\def\fb#1.{\underset #1 {\times}}
\def\rest#1.{\underset {\ \ring #1.} \to \otimes}
\def\au#1.{\operatorname {Aut}\,(#1)}
\def\deg#1.{\operatorname {deg } (#1)}

\def\pic#1.{\operatorname {Pic}\,(#1)}
\def\pico#1.{\operatorname{Pic}^0(#1)}
\def\picg#1.{\operatorname {Pic}^G(#1)}
\def\ner#1.{NS (#1)}
\def\rdown#1.{\llcorner#1\lrcorner}
\def\rfdown#1.{\lfloor{#1}\rfloor}
\def\rup#1.{\ulcorner{#1}\urcorner}
\def\rcup#1.{\lceil{#1}\rceil}

\def\n1#1.{\operatorname {N_1}(#1)}  %Vector space of 1-cycles
\def\cn1#1.{\overline{\operatorname {N^1}(#1)}} %closure
\def\cone#1.{\operatorname {NE}(#1)}     %cone of curves
\def\ccone#1.{\overline{\operatorname {NE}}(#1)}
\def\none#1.{\operatorname {NF}(#1)}
\def\cnone#1.{\overline{\operatorname {NF}}(#1)}
\def\mone#1.{\operatorname {NM}(#1)} %cone of moving curves
\def\cmone#1.{\overline{\operatorname {NM}}(#1)}

\def\coef#1.{\frac{(#1-1)}{#1}}
\def\vit#1.{D_{\langle #1 \rangle}}
\def\mm#1.{\overline {M}_{0,#1}}
\def\H1#1.{H^1(#1,{\ring #1.})}
\def\ac#1.{\overline {\mathbb F}_{#1}}

\def\adj#1.{\frac {#1-1}{#1}}
\def\spn#1.{\overline{#1}}
\def\pek#1.#2.{\Cal P^{#1}(#2)}
\def\plk#1.#2.{\Cal P^{\leq #1}(#2)}
\def\ev#1.{\operatorname{ev_{#1}}}
\def\ilist#1.{{#1}_1,{#1}_2,\dots}
\def\bminv#1.{(\nu_1,s_1;\nu_2,s_2;\dots ;\nu_{#1},s_{#1};\nu_{r+1})}
\def\zinv#1.{(\nu_1,s_1;\nu_2,s_2;\dots ;\nu_{#1},s_{#1};0)}
\def\iinv#1.{(\nu_1,s_1;\nu_2,s_2;\dots ;\nu_{#1},s_{#1};\infty)}

\def\scr #1.{\mathcal #1}

\def\llist#1.#2.{{#1}_1,{#1}_2,\dots,{#1}_{#2}}
\def\ulist#1.#2.{{#1}^1,{#1}^2,\dots,{#1}^{#2}}
\def\lomitlist#1.#2.{{#1}_1,{#1}_2,\dots,\hat {{#1}_i}, \dots, {#1}_{#2}}
\def\lomitlistz#1.#2.{{#1}_0,{#1}_1,\dots,\hat {{#1}_i}, \dots, {#1}_{#2}}
\def\loc#1.#2.{\Cal O_{#1,#2}}
\def\fderiv#1.#2.{\frac {\partial #1}{\partial #2}}
\def\deriv#1.#2.{\frac {d #1}{d #2}}

\def\map#1.#2.{#1 \longrightarrow #2}
\def\rmap#1.#2.{#1 \dasharrow #2}
\def\emb#1.#2.{#1 \hookrightarrow #2}
\def\non#1.#2.{\text {Spec }#1[\epsilon]/(\epsilon)^{#2}}
\def\Hi#1.#2.{\text {Hilb}^{#1}(#2)}
\def\sym#1.#2.{\operatorname {Sym}^{#1}(#2)}
\def\Hb#1.#2.{\text {Hilb}_{#1}(#2)}
\def\Hm#1.#2.{\Hom_{#1}(#2)}
\def\prd#1.#2.{{#1}_1\cdot {#1}_2\cdots {#1}_{#2}}
\def\Bl #1.#2.{\operatorname {Bl}_{#1}#2}
\def\pl #1.#2.{#1^{\otimes #2}}
\def\mgn#1.#2.{\overline {M}_{#1,#2}}
\def\ialist#1.#2.{{#1}_1 #2 {#1}_2, #2\dots}
\def\pair#1.#2.{\langle #1, #2\rangle}
\def\vandermonde#1.#2.{\left|
\begin{matrix}
1 & 1 & 1 & \dots & 1\\
{#1}_1 & {#1}_2 & {#1}_3 & \dots & {#1}_{#2}\\
{#1}_1^2 & {#1}_2^2 & {#1}_3^2 & \dots & {#1}_{#2}^2\\
\vdots & \vdots & \vdots & \ddots & \vdots\\
{#1}_1^{#2-1} & {#1}_2^{#2-1} & {#1}_2^{#2-1} & \dots & {#1}_{#2}^{#2-1}\\
\end{matrix}
\right|
}
\def\vandermondet#1.#2.{\left|
\begin{matrix}
1 & {#1}_1   & {#1}_1^2 & \dots & {#1}_1^{#2-1}\\
1 & {#1}_2   & {#1}_2^2 & \dots & {#1}_2^{#2-1}\\
1 & {#1}_3   & {#1}_3^2 & \dots & {#1}_3^{#2-1}\\
\vdots & \vdots & \vdots & \ddots & \vdots\\
1 & {#1}_{#2}& {#1}_{#2}^2 & \dots & {#1}_{#2}^{#2-1}\\
\end{matrix}
\right|
}
\def\gr#1.#2.{\mathbb{G}(#1,#2)}

\def\alist#1.#2.#3.{{#1}_1 #2 {#1}_2 #2\dots #2 {#1}_{#3}}
\def\zlist#1.#2.#3.{#1_0 #2 #1_1 #2\dots #2 #1_{#3}}
\def\lomitlist30#1.#2.#3.{{#1}_0,{#1}_1 #2 \dots #2\hat {{#1}_i} #2\dots #2 {#1}_{#3}}
\def\lmap#1.#2.#3.{#1 \overset{#2}{\longrightarrow} #3}
\def\mes#1.#2.#3.{#1 \longrightarrow #2 \longrightarrow #3}
\def\ses#1.#2.#3.{0\longrightarrow #1 \longrightarrow #2 \longrightarrow #3 \longrightarrow 0}
\def\les#1.#2.#3.{0\longrightarrow #1 \longrightarrow #2 \longrightarrow #3}
\def\res#1.#2.#3.{#1 \longrightarrow #2 \longrightarrow #3\longrightarrow 0}
\def\Hi#1.#2.#3.{\text {Hilb}^{#1}_{#2}(#3)}
\def\ten#1.#2.#3.{#1\underset {#2}{\otimes} #3}
\def\lomitlist30#1.#2.#3.{{#1}_0 #2 {#1}_1 #2 \dots #2 \hat {{#1}_i} #2 \dots #2 {#1}_{#3}}
\def\mderiv#1.#2.#3.{\frac {d^{#3} #1}{d #2^{#3}}}

\def\Hom{\operatorname{Hom}}

\def\Supp{\operatorname{Supp}}

\def\Exc{\operatorname{Exc}}

\def\dim{\operatorname{dim}}

\def\deg{\operatorname{deg}}

\def\det{\operatorname{det}}

\def\Sing{\operatorname{Sing}}
\def\noninv{\operatorname{non-inv}}

\def\rk{\operatorname{rk}}

\def\rest{\operatorname{res}}

\def\e{\Cal E}

\def\e1{E_1}
\def\e2{E_2}

\usepackage{hyperref,amscd,mathtools}

\title{MMP for smooth co-rank one foliations}


\author{Paolo Cascini}
\address{Department of Mathematics\\
Imperial College London\\
180 Queen's Gate\\
London SW7 2AZ, UK}
\email{p.cascini@imperial.ac.uk}

\author{Calum Spicer}
\address{Department of Mathematics, King's College London, Strand,
London WC2R 2LS, UK}
\email{calum.spicer@kcl.ac.uk}

\begin{document}

\begin{abstract}
We prove the existence of minimal models for smooth co-rank one foliations of general type.
\end{abstract}
\maketitle
\tableofcontents

\section{Introduction}

The goal of this paper is to show that we may run the Minimal Model Program (MMP) starting with a smooth co-rank one foliation 
on a smooth projective variety, and that under mild hypotheses the MMP will terminate in either a minimal model
or a Mori fibre space.  This was proven in \cite{Spicer20} for smooth foliations on threefolds and our main result extends this to arbitrary dimension.

\begin{theorem}[= Corollary \ref{cor_main}]
\label{intro_thm}
Let $X$ be a smooth projective variety and let $\mathcal F$ be a smooth co-rank one foliation on $X$.
Suppose that $K_{\mathcal F}$ is big or that $K_{\mathcal F}$ is not pseudo-effective.

Then we may run a $K_{\mathcal F}$-MMP $\phi\colon X\dashrightarrow X'$ such that, writing $\mathcal F' := \phi_*\mathcal F$, the following hold: 
\begin{enumerate}
\item if $K_{\mathcal F}$ is big, then $\mathcal F'$ is a minimal model, i.e. $K_{\mathcal F'}$ is nef; or
\item if $K_{\mathcal F}$ is not pseudo-effective, then $\mathcal F'$ admits a Mori fibre space, i.e. a contraction  
$\eta\colon X'\to Z$ onto a normal projective variety $Z$ such that $\dim Z<\dim X$, 
$\rho(X'/Z)=1$ and $-K_{\mathcal F'}$ is $\eta$-ample.  
\end{enumerate}
\end{theorem}

Smooth co-rank one holomorphic foliations on projective varieties are rather special and are perhaps amenable to a precise classification.  For example, such a classification has been achieved on surfaces by Brunella, \cite{MR1474805}, and by Touzet, \cite{Touzet08}, in the case that the canonical class of the foliation is trivial.
Theorem \ref{intro_thm} should be viewed as a key first step toward this classification in higher dimensions.  We remark that a smooth foliation on a surface is either a fibration in rational curves or is minimal, i.e. the canonical class of the foliation is nef, and this fact simplifies the classification considerably.  In higher dimensions it is easy to produce examples where this dichotomy fails and therefore a birational simplification of the foliation is necessary in order to achieve a reasonable classification.

An important ingredient in the proof of Theorem \ref{intro_thm} which is also of independent interest is the cone theorem for co-rank one foliations which admit local first integrals.

\begin{theorem}[ = Theorem \ref{t_cone}]
Let $X$ be a normal projective variety, let $\mathcal F$ be a
co-rank one foliation on $X$ which admits local holomorphic first integrals, and
let $\Delta\geq0$ be a $\mathbb Q$-divisor such that
$K_{\mathcal F}+\Delta$ is $\mathbb Q$-Cartier.

Then we may write
\[
\overline{\operatorname{NE}}(X)
=
\overline{\operatorname{NE}}(X)_{K_{\mathcal F}+\Delta\geq0}
+Z_{-\infty}
+\sum_i\mathbb R_{\geq0}[\xi_i],
\]
where each $\xi_i$ is a rational curve tangent to $\mathcal F$
satisfying
\[
0<-(K_{\mathcal F}+\Delta)\cdot\xi_i\leq2\dim X,
\]
and $Z_{-\infty}$ is the closed cone generated by the classes
of curves contained in a non-log canonical centre of
$(\mathcal F,\Delta)$.
\end{theorem}

\subsection*{Sketch of proof} As might be expected, we first prove the cone theorem and then prove existence of minimal models.  The basic strategy for proving the cone theorem is modelled on that of \cite{ACSS}.  The idea is to proceed by induction on dimension.  We first show that the cone theorem in dimension $\leq n-1$ implies the existence of special modifications of the foliation analogous to dlt modifications for varieties in dimension $=n$ (this step makes use of recent developments in the MMP for projective morphisms between complex analytic spaces, cf. \cite{Fuj22,DHP24,LM23}).  We then show that the existence of these special modifications in dimension $=n$ and the cone theorem in dimension $\leq n-1$ imply the cone theorem in dimension $=n$.

One main difficulty in running the MMP for smooth foliations on smooth varieties is that one will typically lose smoothness of the underlying variety (and therefore smoothness of the foliation)
in the course of the MMP.  A more flexible notion of non-singularity of foliations on singular varieties must therefore be found which both has many of the desirable properties of smooth foliations 
on smooth varieties and which is also preserved by the surgery operations of the MMP.  We propose here the class of quasi-smooth foliations and show that this class addresses these requirements. 

A key feature of the quasi-smooth foliations considered here is that the vanishing statement arising from the Bott partial connection on the smooth locus extends across the singular locus. This ensures that for any curve $C$ tangent to $\mathcal F$ we have $K_{\mathcal F}\cdot C = K_X\cdot C$.   Since the cone theorem guarantees that every $K_{\mathcal F}$-negative  extremal ray is spanned by the class of a curve tangent to the foliation, we therefore see that the steps of the $K_{\mathcal F}$-MMP can all be realised as steps in a $K_X$-MMP. 

Showing that quasi-smoothness is preserved by the MMP requires some care.  The key point is to show that the steps of the MMP do not destroy the existence of local holomorphic first integrals of the foliation.  We recall for the reader a theorem due to Malgrange which states that if a co-rank one foliation on a smooth germ $0 \in \mathbb C^n$ has a singular set of codimension $\ge 3$, then the foliation admits a holomorphic first integral.  We utilise a variant of this result on singular spaces based on work of Cerveau and Lins Neto to show that the MMP preserves the existence of local first integrals.

\subsection*{AI Declaration}
All original mathematical ideas in this paper are due to the authors.
ChatGPT~6 was used to assist in checking the arguments and improving
the exposition. The authors take full responsibility for the content
and mathematical correctness of the paper.

\subsection*{Acknowledgements}
We would like to thank S. Druel and J. V. Pereira for useful comments and conversations.

The first author was partially supported by a Simons Collaboration Grant. 
The second author was partially supported by EPSRC.

\section{Preliminaries}

By a variety we mean an irreducible and reduced complex analytic
space. Unless otherwise specified, by a curve we mean an
integral projective curve.

\subsection{Notation}
\label{s_preliminaries}
Let \(X\) be a normal  variety. A {\bf prime divisor} on \(X\)
is an irreducible and reduced closed analytic subvariety of codimension
one. A {\bf divisor} on \(X\) is a formal sum
\[
D=\sum_{i\in I} a_iD_i,
\]
where the \(D_i\) are distinct prime divisors, \(a_i\in\mathbb Z\), and
the sum is locally finite. Equivalently, its support
\[
\operatorname{Supp}D
:=
\bigcup_{a_i\neq0}D_i
\]
is a closed analytic subset of \(X\).  If \(a_i\in\mathbb Q\), respectively
\(a_i\in\mathbb R\), then \(D\) is called a {\bf \(\mathbb Q\)-divisor},
respectively an {\bf \(\mathbb R\)-divisor}. If $\Sigma$ is a prime divisor on $X$,
we denote by $m_\Sigma D$ the coefficient of $D$ along $\Sigma$. 
An \(\mathbb R\)-divisor \(D\) on \(X\) is {\bf \(\mathbb R\)-Cartier
at a point \(x\in X\)} if there exists an open neighbourhood \(U\) of
\(x\) such that \(D|_U\) is a finite \(\mathbb R\)-linear combination
of Cartier divisors on \(U\).  The divisor \(D\) is
{\bf \(\mathbb R\)-Cartier} if it is \(\mathbb R\)-Cartier at every
point of \(X\) (cf. 
\cite[Definition~2.32]{Fuj22}).
Given two $\mathbb R$-divisors $D_1$ and $D_2$ on $X$, we define 
\[
D_1\wedge D_2\coloneqq \sum_{\Sigma} \min \{ m_\Sigma D_1,m_\Sigma D_2\} \Sigma. 
\]

Let \(X\) be a normal variety, let \(Y\) be a smooth analytic curve,
and let
\(
f\colon X\longrightarrow Y
\)
be a dominant holomorphic morphism. A prime divisor \(P\subset X\)
is called {\bf \(f\)-vertical} if \(f(P)\) is a point. For every such \(P\),
let
\[
m_P:=\operatorname{mult}_P f^*(f(P)).
\]
The {\bf ramification divisor} of \(f\) is the effective Weil divisor
\[
R(f):=\sum(m_P-1)P,
\]
where the sum runs over all the $f$-vertical prime divisors. 
Note that the sum is locally finite.

\medskip 
A proper morphism \(f\colon X\to Z\) between normal varieties is
called a {\bf contraction} if the natural morphism
\(
\mathcal O_Z\longrightarrow f_*\mathcal O_X
\)
is an isomorphism.

Let \(\pi\colon X\to Z\) be a projective morphism of normal varieties
and let \(W\subset Z\) be a compact subset. We denote by
\(N_1(X/Z;W)\) the real vector space generated by 
curves \(C\subset X\) such that \(\pi(C)\) is a point of \(W\), modulo
the following numerical equivalence relation:
two real \(1\)-cycles \(\alpha\) and \(\beta\) supported over \(W\)
are {\bf numerically equivalent over \(W\)} if
\(
D\cdot\alpha=D\cdot\beta
\)
for every Cartier divisor \(D\) defined on a neighbourhood of
\(\pi^{-1}(W)\) (we refer to \cite[Section~4.4]{Fuj24} for details). 
Let \(\pi\colon X\to Z\) be a projective morphism of normal
varieties and let \(z\in Z\). 
By Nakayama's finiteness theorem (cf. 
\cite[Theorem~4.5.2]{Fuj24}), the vector space
\(N_1(X/Z;\{z\})
\)
is finite-dimensional. We denote by
\[
\overline{\operatorname{NE}}(X/Z;\{z\})
\subset N_1(X/Z;\{z\})
\]
the closure of the convex cone generated by the classes of curves \(C\subset X\) such that \(\pi(C)=z\).
If \(D\) is an \(\mathbb R\)-Cartier \(\mathbb R\)-divisor on \(X\),
then we denote
\[
D^\perp
\coloneqq
\{\xi\in N_1(X/Z;\{z\})\mid D\cdot\xi=0\}
\]
and 
\[
\overline{\operatorname{NE}}(X/Z;\{z\})_{D\ge 0}\coloneqq
\{
\xi\in \overline{\operatorname{NE}}(X/Z;\{z\})\mid D\cdot\xi\ge 0
\}.
\]

Let \(X\) be a normal projective variety, let \(W\subseteq X\) be an
irreducible subvariety, and let
\(
\nu\colon W^\nu\to W
\)
be its normalisation. Let \(R\subseteq
\overline{\operatorname{NE}}(X)\) be a ray. We say that \(W\)
{\bf carries} \(R\) if
\[
R\subseteq
\operatorname{im}\left(
\overline{\operatorname{NE}}(W^\nu)
\longrightarrow
\overline{\operatorname{NE}}(X)
\right),
\]
where the morphism is induced by the composition
\(W^\nu\to W\hookrightarrow X\).

\medskip

We refer to \cite[Section 2.6]{CS18} for some of the standard definitions in the Minimal Model Program. Let 
$\varphi\colon X\dashrightarrow Y$ be a birational map between normal varieties.  
The {\bf exceptional locus}  of $\varphi$ is the closed subset $\Exc\varphi$ of $X$ 
where $\varphi$ is not an isomorphism. We say that $\varphi$ is  a {\bf birational contraction} if $\Exc \varphi^{-1}$ does not contain any divisor. Let $\varphi\colon X\dashrightarrow Y$ be a birational contraction between normal varieties and let  $D$ be a $\mathbb Q$-Cartier $\mathbb Q$-divisor on $X$ such that $D_Y\coloneqq \varphi_*D$ is also $\mathbb Q$-Cartier. Then $\varphi$ is called {\bf $D$-negative} (resp. {\bf $D$-trivial}) if there exist a normal variety $W$ and birational morphisms $p\colon W\to X$ and $q\colon W\to Y$  resolving the indeterminacy locus of $\varphi$  such that $p^*D=q^*D_Y+E$ where  $E\ge 0$ is a $\mathbb Q$-divisor whose support coincides with the exceptional locus of $q$ 
(resp. $p^*D=q^*D_Y$). Note that by the negativity lemma (cf. \cite[Lemma 3.39]{KM98}), it follows that if $\varphi$ is a composition of steps of a $D$-MMP, then $\varphi$ is $D$-negative.

\medskip
\subsection{Foliations}
Let $X$ be a normal variety. A {\bf foliation of rank $r$} on $X$ consists of a coherent subsheaf $T_\cal F \subset T_X$ of rank $r$ such that
\begin{enumerate}
\item $T_{\cal F}$ is saturated, i.e. $T_X/T_{\cal F}$ is torsion free; and

\item $T_{\cal F}$ is closed under Lie bracket.
\end{enumerate}
The {\bf co-rank} of $\cal F$  is its co-rank as a subsheaf of $T_X$.  
We define the {\bf canonical reflexive sheaf} of $\mathcal F$ by
\(
\omega_{\mathcal F}\coloneqq(\det T_{\mathcal F})^*.
\)
When this sheaf is represented by a Weil divisor, we denote
such a divisor by $K_{\mathcal F}$ and we refer to it as the {\bf canonical divisor} of $\mathcal F$.
On general complex analytic spaces, statements involving
canonical divisors are understood locally, or as statements
about the corresponding canonical reflexive sheaves.
In the algebraic setting, we choose canonical divisors globally.
Moreover, we define the {\bf conormal sheaf} to be $N^{[*]}_{\mathcal F} \coloneqq (T_X/T_{\mathcal F})^*$.
For any positive integer $d$, we denote $\Omega^{[d]}_X\coloneqq (\wedge^d\Omega^1_X)^{**}$.
Then, the foliation $\cal F$ induces a map 
\[
\phi\colon (\Omega^{[r]}_X \otimes \cal O_X(-K_{\cal F}))^{**} \rightarrow \cal O_X.\]
The {\bf singular locus} of $\cal F$, denoted by $\Sing \cal F$, is the cosupport of the image of $\phi$. 
We define the {\bf total singular set} $\Sing(X, \mathcal F) \coloneqq \Sing(X) \cup \Sing(\mathcal F)$.

\medskip

Let $\sigma\colon Y\dashrightarrow X$ be a dominant  map between normal  varieties and let $\cal F$ 
be a foliation of rank $r$ on $X$. We denote by $\sigma^{-1}\cal F$ the {\bf induced foliation} 
on $Y$ (e.g. see \cite[Section 3.2]{Druel21}).  

We refer to \cite[\S2.4]{CS18}  for a recollection on the definition of
foliation singularities from the perspective of the MMP. 
Let $\mathcal F$ be a foliation on a normal variety $X$ and let $E$ be a reduced divisor
on $X$. We denote by
\(
E_{\noninv}
\)
the sum of the prime components of $E$ which are not invariant under
$\mathcal F$.

\subsection{Local first integrals for foliations}

Let $X$ be a normal variety and let $\mathcal F$ be a foliation on $X$.
A holomorphic map $\Phi\colon U\to V$, where $U\subseteq X$ is a 
Euclidean open subset and $V$ is a complex analytic space, is called a
{\bf holomorphic first integral} for $\mathcal F|_U$ if
$T_{\mathcal F}|_U=T_\Phi$.
We say that $\mathcal F$ {\bf admits local holomorphic first integrals}
if, for every $x\in X$, there exist a Euclidean open neighbourhood
$x\in U_x$ and a holomorphic first integral
$\Phi_x\colon U_x\to V_x$ for $\mathcal F|_{U_x}$.

We say that $\mathcal F$ is {\bf quasi-smooth at a point $x \in X$} if

\begin{enumerate}
\item $N^{[*]}_{\mathcal F}$ is locally free at $x$;
\item there exists a Euclidean open neighbourhood $U_x$ and  $f_1, \dots, f_q \in \mathcal O_{U_x}$ where $q = \dim X - \rk \mathcal F$ such that 
$df_1, \dots, df_q$ generate $N^{[*]}_{\mathcal F}|_{U_x}$; and 
\item if we set $\Phi_x:=(f_1, \dots, f_q)\colon U_x \to \mathbb C^q$, then for all $z \in \mathbb C^q$, $\Sing(X, \mathcal F) \cap \Phi_x^{-1}(z)$ is of dimension at most $ \rk \mathcal F-2$.
\end{enumerate}

Note that $\Phi_x$ is a holomorphic first integral for $\mathcal F|_{U_x}$. We say that $\mathcal F$ is {\bf quasi-smooth} if it is quasi-smooth at all points $x \in X$.

\begin{remark}
Let $X$ be a normal variety and let 
$\mathcal F$ be a foliation on $X$. 
\begin{itemize}
\item If $\Sing(X, \mathcal F) = \emptyset$, then $\mathcal F$ is quasi-smooth.
\item Let $X$ be a normal variety, let $Y$ be a smooth variety and consider the foliation $\mathcal F$ on $X \times Y$ induced by the projection $X \times Y \to Y$.  Then $\mathcal F$ is quasi-smooth.
\end{itemize}
\end{remark}

Let $X$ be a normal variety and let $\mathcal F$ be a foliation which admits 
local first integrals.  We say that a subvariety $V \subset X$ is {\bf tangent} to $\mathcal F$
if for a general point $x \in V$ and any first integral $\Phi_x$ 
of $\mathcal F$ in a neighbourhood $U$ of $x$ we have that $\Phi_x(V\cap U)$ is a point.

\begin{remark}
\label{rem_tang_zero}
Note that if $X$ is a variety with rational singularities, $\mathcal F$ is 
a quasi-smooth foliation and $j\colon V \to X$ is a morphism from a variety with rational singularities such that $j(V)$ 
is a subvariety of $X$ which is tangent to $\mathcal F$, then the composition of the natural morphisms
\[j^*N^{[*]}_{\mathcal F} \to j^*\Omega^{[1]}_X \xrightarrow{dj} \Omega^{[1]}_V\]
is identically zero. Here $dj$ is the morphism guaranteed to exist by \cite[Theorem 1.10]{KS21}.
\end{remark}

\begin{definition}[{\cite[Definition~1.1]{CLN08}}]
\label{def_kregular}
Let
\(
(X,x)\subseteq (\mathbb C^N,x)
\)
be an irreducible analytic germ, and let \(0\leq k\leq\dim X\).
We say that \(X\) is {\bf \(k\)-regular at \(x\)} if there exists an
open neighbourhood \(U\subseteq\mathbb C^N\) of \(x\) such that the natural restriction map
\[
H^0(U,\Omega_U^k)
\longrightarrow
H^0\bigl(X_U^\circ,\Omega_{X_U^\circ}^k\bigr)
\]
is surjective, where
\(
X_U\coloneqq X\cap U\)
and
\(
X_U^\circ\coloneqq X_U\setminus\operatorname{Sing}(X_U).
\)
\end{definition}

\begin{remark} \label{r_kregular}
Using the same notations as in Definition \ref{def_kregular}, note that \(0\)-regularity means that every
holomorphic function on \(X_U^\circ\) extends to a holomorphic function
on \(U\). In particular, a normal analytic germ is \(0\)-regular.
\end{remark}

\section{Partial resolutions of foliations admitting local first integrals}

\begin{proposition}
\label{prop_canonical_comparison}
Let \(X\) be a normal analytic variety, let \(Y\) be a smooth analytic
curve, and let
\(
f\colon X\longrightarrow Y
\)
be a dominant holomorphic morphism inducing a co-rank one foliation
\(\mathcal F\) on \(X\). 

Then, 
\[
K_{\mathcal F}
\sim
K_X-f^*K_Y-R(f).
\]
Moreover, let \(\Sigma_Y\subset Y\) be a reduced divisor such that
\(
\Supp R(f)\subseteq f^{-1}(\Sigma_Y)
\),
and set
\(
\Sigma_X:=\bigl(f^{-1}(\Sigma_Y)\bigr)_{\mathrm{red}}
\).

Then
\(
R(f)=f^*\Sigma_Y-\Sigma_X
\),
and therefore
\[
K_{\mathcal F}
\sim
K_X+\Sigma_X-f^*(K_Y+\Sigma_Y).
\]
In particular, these conclusions hold if
\(
f^{-1}(Y\setminus\Sigma_Y)\longrightarrow Y\setminus\Sigma_Y
\)
is smooth.
\end{proposition}

\begin{proof}
Since the statement is determined in codimension one, we may compute
at the generic point of a vertical  prime divisor \(P\subset X\). If \(t\) is a
local parameter at \(f(P)\), then
\[
f^*t=u x^{m_P}
\]
at the generic point of \(P\), where \(u\) is a unit, \(x\) is a
local parameter for \(P\) and $m_P\ge 1$. Thus, the image of
\(
f^*\Omega_Y^1\longrightarrow\Omega_X^{[1]}
\)
vanishes to order \(m_P-1\) along \(P\). Taking reflexive determinants
gives
\[
K_{\mathcal F}
\sim
K_X-f^*K_Y-R(f).
\]

Now suppose that
\(
\Supp R(f)\subseteq f^{-1}(\Sigma_Y)
\).
If \(P\) lies over a point of \(\Sigma_Y\), then its coefficient in
\(f^*\Sigma_Y-\Sigma_X\) is \(m_P-1\). If \(P\) does not lie over
\(\Sigma_Y\), then \(m_P=1\). Thus,
\[
R(f)=f^*\Sigma_Y-\Sigma_X,
\]
as claimed. 

Finally, smoothness over \(Y\setminus\Sigma_Y\) implies that
\(m_P=1\) for every vertical prime divisor \(P\) lying over
\(Y\setminus\Sigma_Y\), and hence our claim follows. 
\end{proof}

\begin{lemma}
\label{l_lc_comparison}
Let \(X\) be a normal analytic variety, let \(Y\) be a smooth analytic
curve, and let
\(
f\colon X\to Y
\)
be a dominant holomorphic morphism inducing a co-rank one foliation
\(\mathcal F\) on \(X\). Let \(B\geq0\) be an \(\mathbb R\)-divisor,
let \(y\in Y\), and set
\(
X_y:=\bigl(f^{-1}(y)\bigr)_{\mathrm{red}}.
\)
Let \(x\in X_y\). Assume that  \((X,X_y+B)\) is log canonical at \(x\). 

Then
\((\mathcal F,B)\) is log canonical at \(x\).
\end{lemma}

\begin{proof}
We work in a neighbourhood of \(x\). In particular, after shrinking
this neighbourhood, every component of \(R(f)\) which meets it is
contained in \(X_y\). By Proposition~\ref{prop_canonical_comparison},
\[
K_{\mathcal F}+B
\sim_{\mathbb R}
K_X+X_y+B-f^*(K_Y+\{y\}).
\]
In particular, \(K_{\mathcal F}+B\) is \(\mathbb R\)-Cartier near \(x\).

 Let
\(
\mu\colon X'\to X
\)
be a birational morphism 
and let \(E\) be a prime divisor on \(X'\) such that \(\mu(E)\) contains \(x\).
Set \(
f':=f\circ\mu
\),
let \(\mathcal F'\) be the foliation induced on \(X'\), and set
\(
X'_y:=\bigl((f')^{-1}(y)\bigr)_{\mathrm{red}}
\).
Proposition~\ref{prop_canonical_comparison} easily implies that 
\[
a(E,\mathcal F,B)
=
a(E,X,X_y+B)+1-\epsilon(E).
\]
Since \((X,X_y+B)\) is log canonical at \(x\), it follows that 
\[
a(E,\mathcal F,B)
\geq
-1+1-\epsilon(E)
=
-\epsilon(E).
\]
Thus \((\mathcal F,B)\) is log canonical at \(x\).
\end{proof}

\begin{proposition}
\label{prop_lc_fibre_adjunction}
Let \(X\) be a normal \(\mathbb Q\)-factorial analytic variety, let
\(Y\) be a smooth analytic curve, and let
\(
f\colon X\longrightarrow Y
\)
be a dominant holomorphic morphism inducing a co-rank one foliation
\(\mathcal F\) on \(X\).
Let \(\Sigma_Y\subset Y\) be a reduced divisor such that
\[
\Supp R(f)\subseteq f^{-1}(\Sigma_Y),
\]
and set
\(
\Sigma_X:=\bigl(f^{-1}(\Sigma_Y)\bigr)_{\mathrm{red}}.
\)
Let \(B\geq0\) be an \(\mathbb R\)-divisor such that
\(
(X,\Sigma_X+B)
\)
is log canonical.
Let \(y\in\Sigma_Y\), let \(T\) be an irreducible component of
\(
\bigl(f^{-1}(y)\bigr)_{\mathrm{red}},
\)
and let
\(
\nu\colon S\longrightarrow T
\)
be the normalisation. Let \(\iota\colon T\hookrightarrow X\) be the
inclusion and \(n:=\iota\circ\nu\). 

Then there exists an effective
\(\mathbb R\)-divisor \(B_S\) such that \((S,B_S)\) is log canonical
and
\[
n^*(K_{\mathcal F}+B)
\sim_{\mathbb R}
K_S+B_S.
\]
\end{proposition}

\begin{proof}
By adjunction, we have
\[
K_S+B_S
\sim_{\mathbb R}
n^*(K_X+\Sigma_X+B),
\]
where
\(
B_S\coloneqq \operatorname{Diff}_S\bigl(\Sigma_X-T+B\bigr)\ge 0
\). 
Since \((X,\Sigma_X+B)\) is log canonical, it follows that \((S,B_S)\) is log canonical.
Since \(f\circ n\) is constant, Proposition \ref{prop_canonical_comparison} implies that  
\[
n^*(K_{\mathcal F}+B)
\sim_{\mathbb R}
n^*(K_X+\Sigma_X+B)
\sim_{\mathbb R}
K_S+B_S,
\]
and our  claim follows.
\end{proof}

\begin{lemma}
\label{l_local_cone}
Let \(\pi\colon X\to Z\) be a projective morphism between normal analytic
varieties, let \(z\in Z\) be a closed point, and assume that  \(X\) is
\(\mathbb Q\)-factorial. Let \(\mathcal F\) be a co-rank one foliation on $X$ and let 
\(B\geq 0\) be a \(\mathbb Q\)-divisor. Suppose, after possibly shrinking \(Z\) around \(z\), that:

\begin{enumerate}
\item \(\mathcal F\) is induced by a holomorphic fibration
\(
f\colon X\longrightarrow Y
\)
onto a smooth analytic curve \(Y\);

\item there exists a holomorphic map \(h\colon Z\to Y\) such that
\( f=h\circ\pi \);  and 

\item if
\(
y:=h(z)\)
and
\(
X_y:=\bigl(f^{-1}(y)\bigr)_{\mathrm{red}}\),
then the pair
\(
(X,X_y+B)
\)
is log canonical on a neighbourhood of \(\pi^{-1}(z)\).
\end{enumerate}

Then
\[
\overline{\operatorname{NE}}(X/Z;\{z\})
=
\overline{\operatorname{NE}}(X/Z;\{z\})_{K_{\mathcal F}+B\geq0}
+
\sum_i\mathbb R_{\geq0}[\xi_i],
\]
where \(\xi_1,\dots\) are rational curves tangent to \(\mathcal F\)
satisfying
\[
0<-(K_{\mathcal F}+B)\cdot\xi_i
\leq 2\operatorname{rk}(\mathcal F).
\]
\end{lemma}

\begin{proof}
By assumption, after possibly shrinking \(Z\) around \(z\), we may
assume that \((X,X_y+B)\) is log canonical. 
In particular, Lemma \ref{l_lc_comparison} implies that \((\mathcal F,B)\) is log canonical. 

Since \(R(f)\) is locally
finite and \(\pi\) is proper, after shrinking \(Z\) further we may
assume that \(\operatorname{Supp}R(f)\) has only finitely many
irreducible components \(P_1,\ldots,P_r\). We may shrink \(Z\) once
more so as to avoid the images under \(\pi\) of those \(P_i\) which do
not meet \(\pi^{-1}(z)\). Every remaining \(P_i\) meets
\(\pi^{-1}(z)\), and, since \(P_i\) is vertical over \(Y\) and
\(f=h\circ\pi\), we have \(f(P_i)=y\).  Thus,
\[
\Supp R(f)\subseteq f^{-1}(y).
\]

By assumption, every  curve \(C\subset X\) such that 
\(\pi(C)=z\) is contained in \(X_y\). Thus, if  
\(T_1,\dots,T_k\) denote the irreducible components of \(X_y\) and 
\(\nu_i\colon S_i\to T_i\) is the normalisation  and 
\(
n_i\colon S_i\longrightarrow X
\)
is the induced morphism, then every  curve in
\(\pi^{-1}(z)\) is dominated by a  curve on \(S_i\), for some \(i=1,\dots,k\).

Fix \(i=1,\dots,k\). We claim that
\(
n_{i*}\overline{\operatorname{NE}}(S_i/Z;\{z\})
\)
is closed. Indeed, let \(H\) be a \(\pi\)-ample Cartier divisor.
Then \(n_i^*H\) is \((\pi\circ n_i)\)-ample. Since
\(N_1(S_i/Z;\{z\})\) is finite-dimensional, for every \(M>0\) the set
\[
\left\{
\alpha\in\overline{\operatorname{NE}}(S_i/Z;\{z\})
\;\middle|\;
(n_i^*H)\cdot\alpha\leq M
\right\}
\]
is compact. The projection formula therefore shows that
\(n_{i*}\overline{\operatorname{NE}}(S_i/Z;\{z\})\) is closed. Since every
curve contracted to \(z\) is dominated by a curve on \(S_i\) for some \(i=1,\dots,k\), it follows that 
\[
\overline{\operatorname{NE}}(X/Z;\{z\})
=\sum_{i=1}^k
n_{i*}\overline{\operatorname{NE}}(S_i/Z;\{z\}).
\]

By  Proposition \ref{prop_lc_fibre_adjunction}, we may write
\[
n_i^*(K_{\mathcal F}+B)\sim_{\mathbb Q} K_{S_i}+B_{S_i},
\]
where \(B_{S_i}\geq0\) and 
\((S_i,B_{S_i})\) is log canonical.

We may therefore apply the analytic cone theorem
\cite[Theorem~1.1.5]{Fuj24} to the projective morphism
\(
\pi\circ n_i\colon S_i\rightarrow Z
\)
over the compact subset \(\{z\}\), for each \(i=1,\dots,k\). 
It follows that the negative extremal
rays are generated by rational curves \(C_{i,1},\dots \) on \( S_i\) satisfying
\[
\overline{\operatorname{NE}}(S_i/Z;\{z\})
=
\overline{\operatorname{NE}}(S_i/Z;\{z\})_{K_{S_i}+B_{S_i}\geq0}
+
\sum \mathbb R_{\geq0}[C_{i,j}]
\]
and
\[
0<-(K_{S_i}+B_{S_i})\cdot C_{i,j}\leq2\dim S_i
=2\operatorname{rk}(\mathcal F).
\]
Let
\(
\xi_{i,j}:=n_i(C_{i,j})\) and 
let \(d_{i,j}\) be the degree of the finite morphism \(C_{i,j}\longrightarrow\xi_{i,j}\). 
Then \(n_{i*}C_{i,j}=d_{i,j}\xi_{i,j}\) and  \(\xi_{i,j}\) is a
rational curve. Moreover, \(f|_{\xi_{i,j}}\) is constant, so \(\xi_{i,j}\) is
tangent to \(\mathcal F\). By the projection formula,
\[
(K_{S_i}+B_{S_i})\cdot C_{i,j}
=
d_{i,j}(K_{\mathcal F}+B)\cdot\xi_{i,j}.
\]
Thus,
\[
0<-(K_{\mathcal F}+B)\cdot\xi_{i,j}
=
\frac{-(K_{S_i}+B_{S_i})\cdot C_{i,j}}{d_{i,j}}
\leq 2\operatorname{rk}(\mathcal F),
\]
and our claim follows.  
\end{proof}

\begin{lemma}
\label{l_local_mmp}
In the setting of Lemma~\ref{l_local_cone}, suppose that \(X\) is klt.

Then every \((K_{\mathcal F}+B)\)-negative extremal ray \(R\) over the
analytic germ \((Z,z)\) admits a contraction \(c_R\colon X\to V\). If the contraction is
small, its flip exists. After every  step of this MMP, all the hypotheses
of Lemma~\ref{l_local_cone} remain satisfied, and the resulting variety
is \(\mathbb Q\)-factorial and klt.
Therefore, every finite sequence of steps of a
\((K_{\mathcal F}+B)\)-MMP over \((Z,z)\) exists.
If \(c_R\) is not birational then it induces a  Mori fibre space.
\end{lemma}

\begin{proof}
Let \(y=h(z)\) and
\(
X_y:=\bigl(f^{-1}(y)\bigr)_{\mathrm{red}}.
\)
As in the proof of Lemma~\ref{l_local_cone}, after possibly shrinking \(Z\)
around \(z\), we may assume that
\[
\Supp R(f)\subseteq X_y,
\]
and Proposition \ref{prop_canonical_comparison} implies that 
\[
K_{\mathcal F}+B
\sim_{\mathbb Q}
K_X+X_y+B-f^*(K_Y+\{y\}).
\]

Let
\(
R\subset\overline{\operatorname{NE}}(X/Z;\{z\})
\)
be a \((K_{\mathcal F}+B)\)-negative extremal ray. By
Lemma~\ref{l_local_cone}, \(R\) is generated by a rational curve
\(\xi\) tangent to \(\mathcal F\). In particular, 
\(
\xi\subseteq X_y
\) 
and 
\[
(K_{\mathcal F}+B)\cdot\xi
=
(K_X+X_y+B)\cdot\xi.
\]
Thus \(R\) is also a
\((K_X+X_y+B)\)-negative extremal ray.

Since \(X\) is klt and \((X,X_y+B)\) is log canonical, for every
sufficiently small rational number \(\delta>0\), the  pair
\(
\bigl(X,(1-\delta)(X_y+B)\bigr)
\)
is klt. Moreover, if  \(\delta\) is sufficiently small, then \(R\) is 
\(
\bigl(K_X+(1-\delta)(X_y+B)\bigr)
\)-negative. Thus, 
after possibly shrinking \(Z\) around \(z\),
the analytic contraction theorem
(cf. \cite[Theorem~1.1.5]{Fuj24})  gives a contraction
\(
c_R\colon X\longrightarrow V
\)
over \(Z\), such that \(\rho(X/V)=1\).
In particular, the contraction \(c_R\) is also over \(Y\) and
\[
K_{\mathcal F}+B
\sim_{c_R}
K_X+X_y+B.
\]

If \(\dim V<\dim X\), then
\(-\left(K_{\mathcal F}+B\right)\) is \(c_R\)-ample and \(c_R\) is a
\((K_{\mathcal F}+B)\)-Mori fibre space.
Suppose therefore that \(c_R\) is birational. If \(c_R\) is
divisorial, set \(X_1:=V\) and let
\(
\varphi:=c_R\colon X\longrightarrow X_1.
\)
If \(c_R\) is small, the corresponding flip exists by
\cite[Theorem~A]{LM23}, applied to the perturbed klt pair. Therefore,
this is also a \((K_{\mathcal F}+B)\)-flip. In this case, let
\(
\varphi\colon X\dashrightarrow X_1
\)
denote the flip.

Let \(\pi_1\colon X_1\to Z\) be the induced morphism and set
\(
f_1:=h\circ\pi_1\), 
\( \mathcal F_1:=\varphi_*\mathcal F\), and
\(B_1:=\varphi_*B\). 
The morphism \(f_1\) induces \(\mathcal F_1\), and
\(
\pi_1^{-1}(z)\subseteq f_1^{-1}(y)
\).
Set
\(
\Sigma_{X_1}:=
\bigl(f_1^{-1}(y)\bigr)_{\mathrm{red}}.
\)
Since \(\varphi\) is birational over \(Y\) and does not extract any
prime divisor, it follows that \(
R(f_1)=\varphi_*R(f)\) and
\(\Sigma_{X_1}=\varphi_*X_y
\).
Thus, \(
{\Supp}R(f_1)\subseteq\Sigma_{X_1}
\).

Since \(\varphi\) is a
\((K_X+X_y+B)\)-negative divisorial contraction or flip,
\(
(X_1,\Sigma_{X_1}+B_1)
\)
remains log canonical. 
Finally, the transformed perturbed pair remains
klt, and hence \(X_1\) is klt and 
\(\mathbb Q\)-factorial.

Thus all the hypotheses of Lemma~\ref{l_local_cone} are preserved
after every step of the MMP, and the argument may be repeated.
\end{proof}

\begin{theorem}
\label{t_adj_model}
Let $X$ be a normal variety, let \(K\subset X\) be a compact subset, and let $(\mathcal F, \Delta)$ be a foliated pair on $X$ where $\mathcal F$ is of co-rank one
such that $\mathcal F$ admits local holomorphic first integrals and $\Delta\ge 0$ is a $\mathbb Q$-divisor.
Let 
\[
\Delta'\coloneqq\Delta \wedge (\sum_{D\subset \Supp \Delta} \epsilon (D) D ). 
\]

Then,
after replacing \(X\) by an open neighbourhood of \(K\), 
 there exist a rational number $t_0 >0$ 
and  a unique projective
birational morphism $\pi\colon X' \rightarrow X$ up to isomorphism over $X$ and 
as a germ over a neighbourhood of \(K\),
such that 
\begin{enumerate}
\item $X'$ is klt; 
\item $(\mathcal F' \coloneqq \pi_*^{-1}\mathcal F, \Gamma \coloneqq \pi_*^{-1}\Delta'+E_{\noninv})$
is a log canonical foliated pair, where $E$ is the sum of the $\pi$-exceptional divisors; 

\item for $0<t < t_0$ we have $a(F, \mathcal F', \Gamma)+ta(F, X') \geq -\epsilon(F)$
for any divisorial valuation $F$;

\item $K_{\mathcal F'}+\Gamma$ is $\pi$-nef;  

\item $K_{\mathcal F'}+\Gamma +tK_{X'}$ is $\pi$-ample for all $0<t <t_0$; and 
\item we may write
\(
K_{\mathcal F'}+\Gamma+G
=
\pi^*(K_{\mathcal F}+\Delta),
\)
where \(G\geq0\) is a \(\mathbb Q\)-divisor on \(X'\) such that
\(
\pi(\Supp G)
\subseteq
\operatorname{Nlc}(\mathcal F,\Delta)
\).
\end{enumerate}
\end{theorem}
\begin{proof}
We first prove uniqueness as a germ over a neighbourhood of \(K\).
Suppose that we have two models $(X_1, \mathcal F_1, \Gamma_1)$ and $(X_2, \mathcal F_2, \Gamma_2)$
satisfying the conditions of the theorem and let   
\(
\pi_i\colon X_i\longrightarrow U_i\), for 
\(i=1,2\), 
be the induced morphisms, where
\(U_1,U_2\subset X\) are open neighbourhoods of \(K\).
After replacing \(U_1\) and \(U_2\) by a common smaller neighbourhood
of \(K\), we may assume that \(U_1=U_2\), and we denote this common
neighbourhood again by \(X\).
Choose a rational number
\(
0<t<\min\{t_1,t_2\},
\)
where \(t_i\) is the constant associated with \(X_i\).
Let $p\colon W \rightarrow X_1$ and $q\colon W \rightarrow X_2$ be birational morphisms which resolve the indeterminacy locus of the  birational map
$X_1 \dashrightarrow X_2$, and let us write $\Gamma_W = p_*^{-1}\Gamma_1+E^1_{\noninv}$ where 
$E^1$ is the $p$-exceptional divisor. Observe that we also have $\Gamma_W = q_*^{-1}\Gamma_2+E^2_{\noninv}$
where $E^2$ is the $q$-exceptional divisor.  Let us also set $\mathcal G \coloneqq p_*^{-1}\mathcal F_1 = q_*^{-1}\mathcal F_2$.
By item (3) above we have 
\[K_{\mathcal G}+\Gamma_W+tK_W  = p^*(K_{\mathcal F_1}+\Gamma_1+tK_{X_1})+F_1\]
and 
\[K_{\mathcal G}+\Gamma_W+tK_W = q^*(K_{\mathcal F_2}+\Gamma_2+tK_{X_2})+F_2\]
where $F_1,F_2\geq0$ are respectively $p$-exceptional and
$q$-exceptional. Set
\[
D_i:=K_{\mathcal F_i}+\Gamma_i+tK_{X_i},
\qquad i=1,2.
\]
After shrinking $X$ around $K$ if necessary, choose a sufficiently
divisible positive integer $r$ such that $rD_i$ is Cartier for
$i=1,2$, and define the graded $\mathcal O_X$-algebras
\[
\mathcal R_i:=\bigoplus_{m\geq0}
(\pi_i)_*\mathcal O_{X_i}(mrD_i).
\]
Since $F_1$ and $F_2$ are effective and exceptional over the
respective models, the above equalities induce a natural
isomorphism
\( \mathcal R_1\simeq\mathcal R_2 \).
On the other hand, item~(5) implies that $D_i$ is $\pi_i$-ample,
and hence
\(X_i\simeq\operatorname{Proj}_X\mathcal R_i\).
Consequently, $X_1$ and $X_2$ are isomorphic over $X$,
which proves uniqueness.

\medskip

We now prove existence over a neighbourhood of a closed point  \(x\in K\).
 After shrinking \(X\) around \(x\), we may assume that \(\mathcal F\)
is induced by a holomorphic first integral
\(
\Phi\colon X\longrightarrow Y
\),
where \(Y\) is a smooth analytic curve. 
Set
\(
y:=\Phi(x)\) and 
\(
X_y:=\bigl(\Phi^{-1}(y)\bigr)_{\mathrm{red}}
\).
Let
\(
\mu\colon\overline X\longrightarrow X
\)
be a log resolution of \((X,\Delta+X_y)\), and set
\(
\overline{\mathcal F}:=\mu^{-1}\mathcal F\), 
\(
\overline\Phi:=\Phi\circ\mu
\), 
\(
\overline X_y:=
\bigl(\overline\Phi^{-1}(y)\bigr)_{\mathrm{red}}
\), 
and 
\(
\overline\Delta
:=
\mu_*^{-1}\Delta'+E_{\mathrm{non\text{-}inv}},
\)
where \(E\) is the reduced \(\mu\)-exceptional divisor. 
Since 
\(
(\overline X,
{\overline X}_y
+\overline\Delta)
\) is log smooth,  after possibly shrinking $X$ around $x$ further, by Lemma \ref{l_lc_comparison} it follows  that 
\(
(\overline{\mathcal F},\overline\Delta)\) is log canonical.

Thus, Lemma \ref{l_local_mmp} implies that 
every contraction and flip occurring in a
\((K_{\overline{\mathcal F}}+\overline\Delta)\)-MMP over the germ
\((X,x)\) exists. 
By Proposition~\ref{prop_canonical_comparison}, this MMP is also a
\(
(K_{\overline X}+\overline X_y+\overline\Delta)
\)-MMP over \(X\). The argument in the proof of the existence of a dlt modification 
(cf. \cite[Theorem 1.8]{Fuj25})
implies 
that the MMP terminates with a model \(\widetilde X\) such that
\(
K_{\widetilde{\mathcal F}}+\widetilde\Delta
\)
is nef over \(X\), where $\psi\colon \overline{X} \dashrightarrow \tilde{X}$ is the induced map, 
$\tilde{\mathcal F} \coloneqq \psi_*\overline{\mathcal F}$ and $\tilde{\Delta} \coloneqq \psi_*\overline{\Delta}$. 

Set
\(
\overline\Theta
:=\overline X_y
+\overline\Delta\)
and 
\(\widetilde\Theta:=\psi_*\overline\Theta.
\)
By construction and by  Lemma 
\ref{l_local_mmp}, we have that 
\(\widetilde X\) is \(\mathbb Q\)-factorial and klt,
\((\widetilde X,\widetilde\Theta)\) is log canonical, and
\[
D
:=
K_{\widetilde{\mathcal F}}+\widetilde\Delta
\sim_{\mathbb Q,X}
K_{\widetilde X}+\widetilde\Theta
\]
is nef over \(X\).

\medskip 

Let \(n:=\dim X\), and choose \(m>0\) such that \(mD\) is Cartier.
Fix a rational number
\(
0<\delta<\frac{1}{2mn+1}.
\)
Since \(\widetilde X\) is klt, we have that the pair
\(
\bigl(\widetilde X,(1-\delta)\widetilde\Theta\bigr)
\)
is klt. Set
\(
L_\delta
:=
K_{\widetilde X}+(1-\delta)\widetilde\Theta
\).
Then
\[
L_\delta
\sim_{\mathbb Q,X}
(1-\delta)D+\delta K_{\widetilde X}.
\]

We run an \(L_\delta\)-MMP over \(X\) with scaling of a sufficiently
positive relatively ample \(\mathbb Q\)-divisor. We claim that every
extremal ray contracted by this MMP is \(D\)-trivial.
Indeed, suppose that
\(
R\subseteq
\overline{\operatorname{NE}}(\widetilde X/X;\{x\})
\)
is an \(L_\delta\)-negative extremal ray and assume that $R$ is $D$-positive.
Then there exists a 
rational curve \(C\) spanning \(R\) such that
\(
0<-K_{\widetilde X}\cdot C\leq 2n.
\)
Since \(mD\) is Cartier and \(D\cdot C>0\), we have
\(
D\cdot C\geq\frac{1}{m}.
\)
Thus,
\[
L_\delta\cdot C\geq
\frac{1-\delta}{m}-2n\delta
>0,
\]
a contradiction.  
Thus, every step of this MMP is a \(D\)-trivial step of a
\(K_{\widetilde X}\)-MMP.

We now show that at each step $X_{i}\dashrightarrow X_{i+1}$ of this MMP, if $D_i$ is the strict transform of $D$ on $X_i$ then $mD_i$ is still Cartier for each $i$. 
Suppose that
\(
c_i\colon X_i\longrightarrow V_i
\)
is the corresponding contraction associated to the extremal ray $R_i$. 
By induction, \(mD_i\) is Cartier and $R_i$ is $D_i$-trivial. 
 Therefore
\cite[Theorem 1.1.5]{Fuj24} implies that there exists a line bundle \(\mathcal M_i\) on \(V_i\) such
that
\(
\mathcal O_{X_i}(mD_i)
\simeq
c_i^*\mathcal M_i.
\)
If \(c_i\) is divisorial, this shows that \(mD_{i+1}\) is Cartier on
\(V_i=X_{i+1}\). If \(c_i\) is small and
\(
c_i^+\colon X_{i+1}\longrightarrow V_i
\)
is its flip, then
\(
\mathcal O_{X_{i+1}}(mD_{i+1})
\simeq
(c_i^+)^*\mathcal M_i.
\)
Thus \(mD_i\) remains Cartier and \(D_i\) remains nef at every step, and we can repeat the argument above. 

Since 
\(L_\delta\) is big over \(X\), 
\cite[Theorem A]{LM23} implies that this
\(L_\delta\)-MMP with scaling terminates.
Denote the resulting MMP by
\(
\phi\colon\widetilde X\dashrightarrow\widetilde X',
\)
let \(g'\colon\widetilde X'\to X\) be the structure morphism, and set
\(
\widetilde{\mathcal F}'
:=
\phi_*\widetilde{\mathcal F},\)
\(
\widetilde\Delta'
:=
\phi_*\widetilde\Delta\),
\(
\widetilde\Theta'
:=
\phi_*\widetilde\Theta\), and 
\(
D'
:=
K_{\widetilde{\mathcal F}'}+\widetilde\Delta'.
\)
Then \(D'\) and
\(
L_\delta'
:=
K_{\widetilde X'}+(1-\delta)\widetilde\Theta'
\)
are nef over \(X\), \(mD'\) is Cartier, and
\[
L_\delta'
\sim_{\mathbb Q,X}
(1-\delta)D'+\delta K_{\widetilde X'}.
\]
Thus, if \(C\) is a curve in $\widetilde X'$ contracted over \(X\) and such that
\(D'\cdot C=0\), then
\(
0
\leq
L_\delta'\cdot C
=
\delta K_{\widetilde X'}\cdot C.
\)
Therefore
\(
K_{\widetilde X'}\cdot C\geq0.
\)

\begin{claim}\label{claim}
For any $0<t< \frac{1}{2m n}$ we have that $K_{\tilde{\mathcal F}'}+\tilde{\Delta}'+tK_{\tilde{X}'}$ is nef over $X$.
Moreover, if $C$ is any curve contracted over $X$ and such that $(K_{\tilde{\mathcal F}'}+\tilde{\Delta}'+tK_{\tilde{X}'})\cdot C = 0$
then $(K_{\tilde{\mathcal F}'}+\tilde{\Delta}')\cdot C = K_{\tilde{X}'}\cdot C = 0$.
\end{claim}

\begin{proof}[Proof of the Claim]
We first claim that
\(
L:=K_{\widetilde X'}+2mnD'
\)
is nef over \(X\). Since \(\widetilde X'\) is klt, we may apply the
relative cone theorem. Let \(R\) be a \(K_{\widetilde X'}\)-negative extremal ray.  Then  \(R\) is generated by a rational curve
\(\xi\) satisfying
\[
0<-K_{\widetilde X'}\cdot\xi\leq 2n.
\]
We want to show that \(D'\cdot\xi>0\). Indeed, if \(D'\cdot\xi=0\), then, 
by construction, we have that 
\(
K_{\widetilde X'}\cdot\xi\geq0,
\)
contradicting the fact that \(R\) is \(K_{\widetilde X'}\)-negative.
Since \(mD'\) is Cartier, it follows that
\(
D'\cdot\xi\geq\frac{1}{m}
\).
Thus,
\[
L\cdot\xi
=
K_{\widetilde X'}\cdot\xi+2mnD'\cdot\xi
\geq
-2n+2mn\frac{1}{m}
=0.
\]
The relative cone theorem therefore implies that \(L\) is nef over
\(X\).

For every
\(
0<t<\frac{1}{2mn}
\),
we can now write
\[
D'+tK_{\widetilde X'}
=
(1-2mnt)D'+tL.
\]
Thus, \(
D'+tK_{\widetilde X'}
\)
is nef over \(X\).

Finally, suppose that \(C\) is a curve contracted over \(X\) such that
\[
(D'+tK_{\widetilde X'})\cdot C=0.
\]
Then \(
((1-2mnt)D'+tL)\cdot C=0\) and, in particular, 
\(
D'\cdot C=L\cdot C=0
\).
Thus, 
\(
K_{\widetilde X'}\cdot C=0
\)
and
\[
\bigl(K_{\widetilde{\mathcal F}'}+\widetilde{\Delta}'\bigr)\cdot C
=
K_{\widetilde X'}\cdot C
=
0,
\]
as required.
\end{proof}

\medskip

Set
\(
\tau:=\frac{1}{2mn}
\) and
for every rational number \(t\) such that \(0<t<\tau\), set
\(
M_t:=D'+tK_{\tilde X'}.
\)
By Claim~\ref{claim}, \(M_t\) is nef over \(X\).
 Since \(\widetilde X'\) is klt, the
analytic base-point-free theorem
\cite[Theorem~8.1]{Fuj22} implies that \(M_t\) is semi-ample over \(X\).
By Claim~\ref{claim}, for every rational \(0<t<\tau\),
the divisors \(M_t\) define the same contraction
\(
p\colon\widetilde X'\longrightarrow X'
\)
over \(X\). Let \(\pi\colon X'\to X\) be the induced morphism. 

Choose two distinct rational numbers \(t_1,t_2\in(0,\tau)\), and write
\(
M_{t_i}\sim_{\mathbb Q}p^*A_i,
\)
where \(A_i\) is \(\pi\)-ample. Since
\[
K_{\widetilde X'}
=
\frac{M_{t_2}-M_{t_1}}{t_2-t_1},
\qquad
D'
=
\frac{t_2M_{t_1}-t_1M_{t_2}}{t_2-t_1},
\]
both \(K_{\widetilde X'}\) and \(D'\) descend through \(p\). Let
\(
\mathcal F'
\) be the foliation induced on \(X'\)
and let
\(
\Gamma:=p_*\widetilde\Delta'
\). Then \(
K_{\widetilde X'}\sim_{\mathbb Q}p^*K_{X'}
\),
and
\(
D'\sim_{\mathbb Q}p^*(K_{\mathcal F'}+\Gamma).
\)
Consequently, \(X'\) is klt and \((\mathcal F',\Gamma)\) is log
canonical. Moreover,
\(
\Gamma=\pi_*^{-1}\Delta'+E_{\mathrm{non\text{-}inv}},
\)
where \(E\) is the reduced \(\pi\)-exceptional divisor. This proves
Properties~{\rm (1)} and {\rm (2)}.

For every rational \(0<t<\tau\), the divisor
\(
K_{\mathcal F'}+\Gamma+tK_{X'}
\)
is \(\pi\)-ample, and, in particular, \(K_{\mathcal F'}+\Gamma\) is
\(\pi\)-nef. This proves Properties~{\rm (4)} and {\rm (5)}.

Since \(\psi\) consists of finitely many
\((K_{\overline{\mathcal F}}+\overline\Delta)\)-negative steps, we
may choose a rational number
\(
0<t_0<\tau
\)
such that every step of \(\psi\) remains
\(
(K_{\overline{\mathcal F}}+\overline\Delta+tK_{\overline X})
\)-
negative for every \(0<t<t_0\). The steps of \(\phi\) are
\((K_{\widetilde{\mathcal F}}+\widetilde\Delta)\)-trivial and
\(K_{\widetilde X}\)-negative, and hence are negative with respect to
\(
K_{\widetilde{\mathcal F}}+\widetilde\Delta+tK_{\widetilde X}
\)
for every \(t>0\). On the other hand, \(p\) is crepant with respect
to both \(K_{\widetilde X'}\) and
\(K_{\widetilde{\mathcal F}'}+\widetilde\Delta'\), and is therefore
trivial with respect to their sum for every \(t>0\).
Thus the rational map
\(\overline X\dashrightarrow X'\) is
\((
K_{\overline{\mathcal F}}+\overline\Delta+tK_{\overline X})
\)-
non-positive for every \(0<t<t_0\). Since
\[
a(F,\overline{\mathcal F},\overline\Delta)
+t\,a(F,\overline X)\geq-\epsilon(F),
\]
discrepancies do not decrease under this map, and Property~{\rm (3)}
follows.
Property (6) is a direct consequence of (4) and the Negativity Lemma.

\medskip 

Thus, for every $x\in K$, 
we have  obtained an open
neighbourhood \(U_x\) of \(x\), a projective birational morphism
\(
\pi_x\colon X'_x\longrightarrow U_x,
\)
and a rational number \(t_x>0\) such that Properties
\textup{(1)--(6)} hold over \(U_x\) for every \(0<t<t_x\).
We now globalise the construction over \(K\).
For every \(x\in K\), choose open neighbourhoods
\(
V_x\subset U_x
\) of $x$ 
such that \(\overline{V_x}\) is compact and contained in \(U_x\).
By compactness, there exist points \(x_1,\ldots,x_r\in K\) such that
\[
K\subseteq V_{x_1}\cup\cdots\cup V_{x_r}.
\]
Choose a rational number
\[
0<t_0<\min_{1\leq i\leq r}t_{x_i}.
\]

The uniqueness proved above gives a unique isomorphism
\[
X'_{x_i}|_{V_{x_i}\cap V_{x_j}}
\simeq
X'_{x_j}|_{V_{x_i}\cap V_{x_j}}
\]
over \(V_{x_i}\cap V_{x_j}\).  Thus, the local models glue
to a normal analytic variety \(X'\) and a proper birational morphism
\(
\pi\colon X'\longrightarrow V,
\) where
\(V:=V_{x_1}\cup\cdots\cup V_{x_r}.
\)

It remains to verify that \(\pi\) is projective.  Fix a rational number
\(0<t<t_0\) and set
\[
M_{t}:=K_{\mathcal F'}+\Gamma+tK_{X'}.
\]
For each $i$, let \(\pi_{x_i}\colon X'_{x_i}\longrightarrow U_{x_i}\)
be the original local model, and let $M_{t,i}$ denote the
corresponding $\mathbb Q$-Cartier divisor on $X'_{x_i}$.
Since $\overline{V_{x_i}}$ is compact and contained in $U_{x_i}$,
and $\pi_{x_i}$ is proper, the set
\(
\pi_{x_i}^{-1}\bigl(\overline{V_{x_i}}\bigr)
\)
is compact. Hence there exists an integer $q_i>0$ such that
$q_iM_{t,i}$ is Cartier on a neighbourhood of this set.

Choose a common multiple $q>0$ of the integers $q_1,\ldots,q_r$.
Under the identifications used to glue the local models,
$qM_t$ is Cartier on each $\pi^{-1}(V_{x_i})$, and therefore
on all of $X'$.
Moreover, by the local construction,
\[
\mathcal O_{X'}(qM_t)|_{\pi^{-1}(V_{x_i})}
\]
is ample relative to $\pi^{-1}(V_{x_i})\to V_{x_i}$ for every $i$.
Since relative ampleness is local on the base,
$\mathcal O_{X'}(qM_t)$ is $\pi$-ample.
Thus $\pi$ is projective.

Properties \textup{(1), (2), (4)} and \textup{(6)} are local on the
base.  Properties \textup{(3)} and \textup{(5)} hold on every
\(V_{x_i}\) for all \(0<t<t_0\), by the choice of \(t_0\).
They therefore hold over \(V\).  Replacing \(X\) by \(V\), the theorem follows.
\end{proof}

\section{Cone and Contraction theorems}

\subsection{Cone theorem for foliations admitting local first integrals}

\begin{lemma}
\label{l_only_lc_extracted}
Let \(X\) be a normal projective variety, let \(\mathcal F\) be a
co-rank one foliation on \(X\) which admits local holomorphic first
integrals, and let \(B,B'\geq0\) be effective \(\mathbb Q\)-divisors.
Assume that \(K_{\mathcal F}+B\) and  \(B'\) are
\(\mathbb Q\)-Cartier. For every sufficiently small rational number
\(t>0\), let
\(
\mu_t\colon Y_t\longrightarrow X
\)
be the modification of \((\mathcal F,B+tB')\) whose existence is
guaranteed by Theorem~\ref{t_adj_model}.

Then there exists a constant \(\alpha>0\) such that, for every rational
number \(0<t<\alpha\), if \(E\) is a \(\mu_t\)-exceptional prime
divisor, then
\[
a(E,\mathcal F,B)\leq-\epsilon(E).
\]
\end{lemma}

\begin{proof}
Fix a  sufficiently small rational number \(s>0\). By the construction in the proof of Theorem~\ref{t_adj_model}, whenever
\(0<t<s\), there is a birational contraction
\(
Y_{s}\dashrightarrow Y_t
\)
over \(X\). Thus, 
we may assume that 
there are finitely many divisorial valuations 
\(
E_1,\ldots,E_r
\)
such that every \(\mu_t\)-exceptional prime divisor is represented by
one of the \(E_i\).

Set
\[
a_i:=a(E_i,\mathcal F,B)
\qquad\text{and}\qquad
b_i:=\operatorname{ord}_{E_i}(B').
\]
Since \(B'\geq0\) is \(\mathbb Q\)-Cartier, we have \(b_i\geq0\).
Let
\[
I:=
\left\{
i\ \middle|\
a_i>-\epsilon(E_i)\ \text{and}\ b_i>0
\right\}.
\]
If \(I\) 
is empty, set \(\alpha:=s\). Otherwise choose
\[
0<\alpha<
\min\left\{
s,
\frac{a_i+\epsilon(E_i)}{b_i}
\ \middle|\
i\in I
\right\}.
\]

Let \(0<t<\alpha\), and let \(E\) be a
\(\mu_t\)-exceptional prime divisor. Then \(E\) is one of the
valuations \(E_i\). By Property~{\rm (6)} of
Theorem~\ref{t_adj_model}, we may write
\[
K_{\mathcal F_{Y_t}}+\Gamma_t+G_t
=
\mu_t^*(K_{\mathcal F}+B+tB'),
\]
for some \(
G_t\geq0\).
Since the coefficient of \(E\) in \(\Gamma_t\) is
\(\epsilon(E)\), we have 
\[
a_i-tb_i= a(E_i,\mathcal F,B+tB')
\leq-\epsilon(E_i).
\]

Suppose that \(a_i>-\epsilon(E_i)\). Then  \(b_i>0\), and hence \(i\in I\). On the other hand, by the choice of
\(\alpha\), we have 
\[
a_i-tb_i>-\epsilon(E_i),
\]
which is a contradiction. Therefore
\[
a(E_i,\mathcal F,B)=a_i\leq-\epsilon(E_i),
\]
as required.
\end{proof}

\begin{theorem}
\label{t_cone}
Let $X$ be a normal projective variety, let $\mathcal F$ be a
co-rank one foliation on $X$ which admits local holomorphic first integrals, and
let $\Delta\geq0$ be a $\mathbb Q$-divisor such that
$K_{\mathcal F}+\Delta$ is $\mathbb Q$-Cartier.

Then we may write
\[
\overline{\operatorname{NE}}(X)
=
\overline{\operatorname{NE}}(X)_{K_{\mathcal F}+\Delta\geq0}
+Z_{-\infty}
+\sum_i\mathbb R_{\geq0}[\xi_i],
\]
where each $\xi_i$ is a rational curve tangent to $\mathcal F$
satisfying
\[
0<-(K_{\mathcal F}+\Delta)\cdot\xi_i\leq2\dim X,
\]
and $Z_{-\infty}$ is the closed cone generated by the classes
of curves contained in a non-log canonical centre of
$(\mathcal F,\Delta)$.
\end{theorem}

\begin{proof}
We proceed by induction on $n\coloneqq\dim X$ and set
\(
D\coloneqq K_{\mathcal F}+\Delta
\).
The claim is immediate if $n=1$.

By Theorem~\ref{t_adj_model},  there exists a projective birational morphism
\(
\pi\colon X'\to X
\)
such that $X'$ is klt and, if
\(
\mathcal F'\coloneqq\pi_*^{-1}\mathcal F
\),
\[
 \Delta'
\coloneqq
\Delta\wedge
\left(
\sum_{P\subseteq\operatorname{Supp}\Delta}\epsilon(P)P
\right),
\]
and
\(
\Gamma\coloneqq
\pi_*^{-1} \Delta'+E_{\mathrm{non\text{-}inv}},
\)
where $E$ is the reduced $\pi$-exceptional divisor, then
\[
K_{\mathcal F'}+\Gamma+G
=
\pi^*(K_{\mathcal F}+\Delta)
\]
for some effective $\mathbb Q$-divisor $G$ such that
\(
\pi(\operatorname{Supp}G)
\subseteq
\operatorname{Nlc}(\mathcal F,\Delta)
\).

Let
\(
\tau\colon X''\to X'
\)
be a small  $\mathbb Q$-factorialisation, whose existence
follows from \cite[Corollary~1.4.3]{BCHM10}. Set
\(
\mu\coloneqq\pi\circ\tau
\),
\(
\mathcal F''\coloneqq\tau_*^{-1}\mathcal F'
\), and
\(
\Delta''\coloneqq
\tau_*^{-1}\Gamma+\tau_*^{-1}G
\).
Since $\tau$ is small, we have
\[
K_{\mathcal F''}+\Delta''
=
\mu^*(K_{\mathcal F}+\Delta).
\]
Moreover, $X''$ is  klt and $\mathbb Q$-factorial, and
$\Delta''\geq0$.

It is enough to prove the theorem for
$(X'',\mathcal F'',\Delta'')$. Indeed,
\(
\mu_*\overline{\operatorname{NE}}(X'')
=
\overline{\operatorname{NE}}(X)
\),
and crepantness implies that the centre on $X$ of every non-log
canonical place of $(\mathcal F'',\Delta'')$ is a non-log canonical
centre of $(\mathcal F,\Delta)$. Moreover, a
$(K_{\mathcal F''}+\Delta'')$-negative curve is not contracted by
$\mu$. If $C''$ is such a curve on $X''$, $C=\mu(C'')$, and $d$ is the
degree of $C''\to C$, then
\[
-(K_{\mathcal F}+\Delta)\cdot C
=
\frac{
 -(K_{\mathcal F''}+\Delta'')\cdot C''
}{d}.
\]
Thus, replacing $(X,\mathcal F,\Delta)$ by
$(X'',\mathcal F'',\Delta'')$, we may assume that $X$ is 
klt and $\mathbb Q$-factorial.

\medskip 

Let
\(
R\subset\overline{\operatorname{NE}}(X)
\)
be a $D$-negative exposed extremal ray which is not contained in
$Z_{-\infty}$. Let $A$ be an ample $\mathbb R$-divisor such that
\(
H_R\coloneqq D+A
\)
is nef and
\(
\overline{\operatorname{NE}}(X)\cap H_R^\perp=R
\).

Suppose first that $\nu(H_R)<n$. Set
\(
\nu\coloneqq\nu(H_R)
\),
let
\(
D_i=H_R
\)
for $1\leq i\leq\nu+1$, and let
\(
D_i=A
\)
for $\nu+1<i\leq n$. Then
\(
D_1\cdots D_n=0
\)
and
\(
-D\cdot D_2\cdots D_n>0
\).
Fix $k\gg0$ such that
\(
M\coloneqq kH_R-D
\)
is ample. Arguing as in the proof of
\cite[Lemma~3.12]{ACSS}, and applying
\cite[Corollary~2.28]{Spicer20}, we obtain a rational curve $\xi$
tangent to $\mathcal F$ such that $H_R\cdot\xi=0$ and
\[
\begin{aligned}
0<-D\cdot\xi=M\cdot\xi
&\leq
2n
\frac{
 M\cdot H_R^\nu\cdot A^{n-\nu-1}
}{
 -K_{\mathcal F}\cdot H_R^\nu\cdot A^{n-\nu-1}
}\\
&=
2n
\frac{
 -D\cdot H_R^\nu\cdot A^{n-\nu-1}
}{
 -K_{\mathcal F}\cdot H_R^\nu\cdot A^{n-\nu-1}
}
\leq2n.
\end{aligned}
\]
Since 
\(
\overline{\operatorname{NE}}(X)\cap H_R^\perp=R
\),
the curve $\xi$ spans $R$.

\medskip 

Suppose now that $H_R$ is big. We may write
\(
H_R\equiv A_0+G_0
\),
where $A_0$ is ample and $G_0\geq0$. Since $H_R\cdot R=0$, we have
$G_0\cdot R<0$. We may therefore choose a prime component
$G\subseteq\operatorname{Supp}G_0$ such that
\(
G\cdot R<0
\).
Since $X$ is $\mathbb Q$-factorial, $G$ is $\mathbb Q$-Cartier.
In particular $G$ carries $R$.
Since
$R\not\subseteq Z_{-\infty}$, the pair
$(\mathcal F,\Delta)$ is log canonical at the generic point of every
subvariety which carries $R$ (cf. \S \ref{s_preliminaries}).

For every $W\subseteq G$ which carries $R$, let
\[
c_W
\coloneqq
\sup\left\{
c\geq0
\ \middle|\
(\mathcal F,\Delta+cG)
\text{ is log canonical at the generic point of }W
\right\}.
\]

\begin{claim}
The numbers $c_W$ form a finite set of rational numbers and 
\((\mathcal F,\Delta+c_WG)\) is log canonical at the generic point of \(W\). 
\end{claim}

\begin{proof}[Proof of the Claim]
Choose finitely many relatively compact open subsets $V_i$ covering
$X$, where each $V_i$ is contained in an open subset $U_i$ on which
$\mathcal F$ admits a holomorphic first integral \(\Phi_i\colon U_i\to Z_i\) onto a smooth analytic curve $Z_i$.
After shrinking $U_i$ further, we may assume that $\Supp R(\Phi_i)$ has only finitely many irreducible components. 
 On each $U_i$,
resolve simultaneously the first integral and
$\operatorname{Supp}(\Delta+G)$ over a neighbourhood of
$\overline{V_i}$. 

Lemma \ref{l_lc_comparison} implies that 
  foliated log canonicity
on $V_i$ coincides with ordinary log canonicity after adding the corresponding
reduced fibre of the first integral. It follows that every finite
threshold $c_W$ with $W\cap V_i\neq \emptyset$ is computed by one of the finitely many prime
divisors $P$ on this fixed model and is of the form
\[
\frac{
 a(P,\mathcal F,\Delta)+\epsilon(P)
}{
 \operatorname{ord}_P(G)
}
\]
for some \(P\) such that \( \operatorname{ord}_P(G)>0\). 
The claim then follows easily.
\end{proof}

Set
\[
\lambda_0
\coloneqq
\min\left\{
c_W
\ \middle|\
W\subseteq G\text{ carries }R
\right\},
\]
and choose $W_0$ maximal with respect to inclusion among the
subvarieties which carry $R$ and satisfy $c_{W_0}=\lambda_0$.
In particular, $W_0$ is a log canonical centre of
$(\mathcal F,\Delta+\lambda_0G)$. 

Set
\(
\Theta_0\coloneqq\Delta+\lambda_0G
\).
Note that $(\mathcal F,\Theta_0)$ is log canonical at the
generic point of every subvariety $V\subseteq G$ which carries $R$,
since $c_V\geq\lambda_0$.

Let $\alpha>0$ be the constant given by
Lemma~\ref{l_only_lc_extracted}, applied to
$(\mathcal F,\Theta_0)$ and $G$. Choose a sufficiently small rational
number
\(
0<\delta<\alpha
\)
such that there is no threshold $c_W$ in
\(
(\lambda_0,\lambda_0+\delta]
\).
Set
\(
\Theta_\delta\coloneqq
\Delta+(\lambda_0+\delta)G
\),
and let
\(
\pi_\delta\colon X_\delta\to X
\)
be the modification of $(\mathcal F,\Theta_\delta)$ given by
Theorem~\ref{t_adj_model}. Let
\(
\tau_\delta\colon Y\to X_\delta
\)
be a small  $\mathbb Q$-factorialisation and set
\(
\mu\coloneqq\pi_\delta\circ\tau_\delta
\). Let $\mathcal F_Y$ be the induced foliation on $Y$. 

For $t\in\{0,\delta\}$, set
\[
\Theta_t
\coloneqq
\Delta+(\lambda_0+t)G,
\qquad
\Theta_t'
\coloneqq
\Theta_t\wedge
\left(
\sum_{P\subseteq\operatorname{Supp}\Theta_t}
\epsilon(P)P
\right),
\]
and
\(
\Gamma_t
\coloneqq
\mu_*^{-1}\Theta_t'
+(E_\mu)_{\mathrm{non\text{-}inv}},
\)
where $E_\mu$ is the reduced $\mu$-exceptional divisor.
Property~{\rm (6)} of Theorem~\ref{t_adj_model} gives
\[
K_{\mathcal F_Y}+\Gamma_\delta+F_\delta
=
\mu^*(K_{\mathcal F}+\Theta_\delta)
\]
for some $F_\delta\geq0$. Define $F_0$ by
\[
K_{\mathcal F_Y}+\Gamma_0+F_0
=
\mu^*(K_{\mathcal F}+\Theta_0).
\]
Since $\Gamma_0\leq\Gamma_\delta$, 
Property~{\rm (2)} of Theorem~\ref{t_adj_model} implies that 
$(\mathcal F_Y,\Gamma_0)$ 
and $(\mathcal F_Y,\Gamma_\delta)$ 
are log canonical.
Moreover, Lemma~\ref{l_only_lc_extracted} implies that $F_0\geq0$ and
\[
\mu(\operatorname{Supp}F_0)
\subseteq
\operatorname{Nlc}(\mathcal F,\Theta_0).
\]

Since $W_0$ is a log canonical centre of
$(\mathcal F,\Theta_0)$ and $W_0\subseteq G$, the pair
$(\mathcal F,\Theta_\delta)$ is not log canonical at the generic point of $W_0$.
Since $(\mathcal F_Y,\Gamma_\delta)$ is log canonical, there exists
a component $E$ of $F_\delta$ such that
\(
W_0\subseteq\mu(E)
\).
On the other hand, $(\mathcal F,\Theta_0)$ is log canonical at the generic point of 
$W_0$, and hence
\(
E\not\subseteq\operatorname{Supp}F_0
\). 
Thus, if
\(
T\coloneqq\mu(E)
\),
then
\(
W_0\subseteq T\subseteq G
\), and $T$ carries $R$.

It follows that $(\mathcal F,\Theta_\delta)$ is not log canonical at
the generic point of $T$ and \(
\lambda_0\leq c_T<\lambda_0+\delta
\).
The choice of $\delta$ implies that $c_T=\lambda_0$, and the
maximality of $W_0$ gives $T=W_0$.
Moreover, $E$ is a log canonical place of
$(\mathcal F,\Theta_0)$. 

Let
\(
\nu\colon S\to E
\)
be the normalisation, let
\(
j\colon S\to Y
\)
be the induced morphism, and set
\(
\rho\coloneqq\mu\circ j\colon S\to X
\).
By foliation adjunction, \cite{CS23b}, there exists a foliation $\mathcal G$ on $S$ and a $\mathbb Q$-divisor $\Gamma_S$ on $S$ such that
\[
j^*(K_{\mathcal F_Y}+\Gamma_0)
=
K_{\mathcal G}+\Gamma_S,
\]
where $\Gamma_S\geq0$. If $E$ is not invariant, then
$(\mathcal G,\Gamma_S)$ is a log canonical foliated pair,
$\mathcal G$ has co-rank one and admits local first integrals, by
\cite[Theorem~3.16]{CS23b}. If $E$ is invariant, then
$\mathcal G=T_S$ and $(S,\Gamma_S)$ is log canonical, cf. Proposition \ref{prop_lc_fibre_adjunction}.

Since $Y$ is $\mathbb Q$-factorial, $F_0$ is $\mathbb Q$-Cartier.
Moreover, $E\not\subseteq\operatorname{Supp}F_0$, and therefore
\(
F_S\coloneqq j^*F_0\geq0
\).
We have
\[
K_{\mathcal G}+\Gamma_S+F_S
=
\rho^*(K_{\mathcal F}+\Theta_0).
\]

Since $W_0$ carries $R$, the face
\[
\mathcal C
\coloneqq
\overline{\operatorname{NE}}(S)
\cap(\rho^*H_R)^\perp
\]
contains a class which is not contracted by $\rho$. 
Thus, there exists an extremal ray
\(
R_S\subseteq\overline{\operatorname{NE}}(S)
\)
which is not contracted by $\rho$ and such that
\(
\rho_*R_S=R
\).

We claim that
\(
F_S\cdot R_S\geq0
\).
Otherwise,  some
prime component $P$ of $F_S$ carries $R_S$. Then $\rho(P)$ carries
$R$ and
\(
\rho(P)
\subseteq
\operatorname{Nlc}(\mathcal F,\Theta_0)
\).
On the other hand,
\(
\rho(P)\subseteq\rho(S)=W_0\subseteq G
\),
contradicting the fact that $(\mathcal F,\Theta_0)$ is log canonical
at the generic point of every subvariety of $G$ which carries $R$.

Since $D\cdot R<0$ and $G\cdot R<0$, we have
\(
(K_{\mathcal F}+\Theta_0)\cdot R<0
\).
It follows that
\(
(K_{\mathcal G}+\Gamma_S)\cdot R_S<0
\).
If $E$ is not invariant, the induction hypothesis gives a rational
curve $\xi'$ tangent to $\mathcal G$, spanning $R_S$, such that
\[
0<
-(K_{\mathcal G}+\Gamma_S)\cdot\xi'
\leq
2\dim S
\leq2n.
\]
If $E$ is invariant, the classical cone theorem applied to the log
canonical pair $(S,\Gamma_S)$ gives the same conclusion. In this
case every curve on $S$ is tangent to
\(
\mathcal G=T_S
\).

Since
\(
F_S\cdot R_S\geq0
\)
and
\(
(K_{\mathcal G}+\Gamma_S+F_S)\cdot R_S<0
\),
in either case we have
\[
0<
-(K_{\mathcal G}+\Gamma_S+F_S)\cdot\xi'
\leq2n.
\]

The curve $\xi'$ is not contracted by $\rho$. Let
\(
\xi\coloneqq\rho(\xi')
\)
and let $d$ be the degree of $\xi'\to\xi$. Then $\xi$ is a rational
curve spanning $R$. By the projection formula and the inequality
$G\cdot\xi<0$, we have
\[
\begin{aligned}
0<-D\cdot\xi
&=-(D+\lambda_0G)\cdot\xi+\lambda_0G\cdot\xi\\
&\leq-(D+\lambda_0G)\cdot\xi\\
&=
\frac{
 -(K_{\mathcal G}+\Gamma_S+F_S)\cdot\xi'
}{d}
\leq2n.
\end{aligned}
\]

Finally, $\xi$ is tangent to $\mathcal F$.
Choose a general point $x'\in\xi'$ and set $x:=\rho(x')$.
We may assume that $\xi'$ and $\xi$ are smooth at $x'$ and $x$,
respectively, and that $\rho|_{\xi'}$ is a local isomorphism
at $x'$.

Let $\Phi_x\colon U\to Z_x$ be a holomorphic first integral
for $\mathcal F$ on an open neighbourhood $U$ of $x$, and
choose a sufficiently small connected open neighbourhood
$U'\subseteq\rho^{-1}(U)$ of $x'$.
If $E$ is not invariant, then
\(
(\Phi_x\circ\rho)|_{U'}
\)
is a holomorphic first integral for $\mathcal G|_{U'}$.
Since $\xi'$ is tangent to $\mathcal G$, this map is constant
along $\xi'$ near $x'$.
If $E$ is invariant, then $(\Phi_x\circ\rho)|_{U'}$ is constant
on $U'$, so the same conclusion holds.
Since $\rho|_{\xi'}$ is a local isomorphism at $x'$,
it follows that $\Phi_x$ is constant along $\xi$ near $x$.
Thus $\xi$ is tangent to $\mathcal F$, as required.

Thus, we have proved that every $D$-negative exposed extremal ray not
contained in $Z_{-\infty}$ is spanned by a rational curve $\xi$ tangent to
$\mathcal F$ such that 
$0<-(K_{\mathcal F}+\Delta)\cdot \xi \leq 2\dim X$.
Standard arguments then apply to show that the Kleiman-Mori cone has the claimed form.
\end{proof}

\subsection{Existence of contractions}

We recall the following result due to Araujo-Druel. 

\begin{lemma}
\label{lem_coho}
Let $X$ be a variety, let $q \ge 0$ be an integer and let $\mathcal E$ be a coherent sheaf on $X$ which satisfies Serre's condition $S_{q+2}$ and such that $\Supp \mathcal E=X$.
Let $U \subset X$ be an open subset such that $X \setminus U$ is of codimension at least $q+2$ in $X$.

Then $H^i(X, \mathcal E) \simeq H^i(U, \mathcal E|_U)$ for all $0 \leq i \leq q$
and $H^{q+1}(X, \mathcal E) \to H^{q+1}(U, \mathcal E|_U)$ is injective.
\end{lemma}
\begin{proof}
This follows from \cite[Lemma 6.4]{AD14}.
\end{proof}

This has the following corollary, cf. \cite[Corollary 2.3]{DPPT26}.

\begin{proposition}
\label{p_bott_vanishing}
Let $X$ be a variety with rational singularities, let $\mathcal F$ be a quasi-smooth co-rank one foliation on $X$, 
and suppose that $\Sing(X, \mathcal F)$ has codimension at least three in $X$.
Let $C$ be a  curve tangent to $\mathcal F$.  

Then $c_1(N^{[*]}_{\mathcal F})\cdot C = 0$.
\end{proposition}
\begin{proof}
Let $U = X \setminus \Sing(X, \mathcal F)$.  Recall that $c_1(N^{[*]}_{\mathcal F}|_U)$ lies in the image of 
\[
H^1(U, N^{[*]}_{\mathcal F}|_U) \to H^1(U, \Omega^1_U),
\]
(cf. \cite[Lemma 2.2]{DPPT26}).
Since $X$ has rational singularities and $N^{[*]}_{\mathcal F}$ is locally free, it is also $S_k$ for all $k \ge 0$, hence Lemma \ref{lem_coho} implies that 
the natural morphism
\begin{align}
\label{iso1}
H^1(X, N^{[*]}_{\mathcal F}) \to H^1(U, N^{[*]}_{\mathcal F}|_U)
\end{align}
is an isomorphism and
as $\Omega^{[1]}_X$ is $S_2$, Lemma \ref{lem_coho} implies that the natural morphism 
\begin{align}\label{inj}
H^1(X, \Omega^{[1]}_X) \to H^1(U, \Omega^1_U)
\end{align}
is injective.
From the isomorphism \eqref{iso1} and the injection \eqref{inj} we see that $c_1(N^{[*]}_{\mathcal F})$ lies in the image of the map $s\colon H^1(X, N^{[*]}_{\mathcal F}) \to H^1(X, \Omega^{[1]}_X)$.

Let $\tilde C \to C$ be the normalisation and let $j\colon \tilde C \to X$ be the composed map.  Since $X$ has rational singularities we have a natural morphism $j^*\Omega^{[1]}_X \to \Omega^1_{\tilde C}$, \cite[Theorem 1.10]{KS21}  which induces a map
$r\colon H^1(X, \Omega^{[1]}_X) \to H^1(\tilde C, \Omega_{\tilde C}^1)$ and for any line bundle $L$ on $X$ we have that $c_1(j^*L)$ lies in the image of $r$.

Thus, $c_1(j^*N^{[*]}_{\mathcal F})$ is in the image of the morphism $r\circ s\colon H^1(X, N^{[*]}_{\mathcal F}) \to H^1(\tilde C, \Omega_{\tilde C}^1)$.
Since $C$ is tangent to $\mathcal F$, by Remark \ref{rem_tang_zero}
 the morphism $j^*N^{[*]}_{\mathcal F} \to \Omega^1_{\tilde C}$ is the zero morphism, hence $c_1(j^*N^{[*]}_{\mathcal F}) = 0$, as required. 
\end{proof}

\begin{theorem}
\label{t_contraction}
Let $X$ be a $\mathbb Q$-factorial projective variety with rational singularities and let $\mathcal F$ be a quasi-smooth co-rank one foliation on $X$, 
and suppose that $\Sing(X, \mathcal F)$ has codimension at least three in $X$.
Let $B \geq 0$ be a $\mathbb Q$-divisor on $X$ such that $(\mathcal F, B)$ is log canonical
and $(X, B)$ is log canonical. Let $R$ be a $(K_{\mathcal F}+B)$-negative extremal ray.

Then there exists a contraction $c_R\colon X \rightarrow Y$ such that for all curves $\Sigma$ on $X$, we have that 
$c_R(\Sigma)$ is a point if and only if $[\Sigma] \in R$.
Moreover, if $c_R$ is a small contraction, then the $(K_{\mathcal F}+B)$-flip exists.
\end{theorem}
\begin{proof}
By Theorem \ref{t_cone} we have $R = \mathbb R_+[\xi]$ where $\xi$
is a curve on $X$ tangent to $\mathcal F$.
Since $\mathcal F$ is quasi-smooth and $\xi$ is tangent to $\mathcal F$, 
Proposition \ref{p_bott_vanishing} 
guarantees that  
$K_{\mathcal F}\cdot \xi= K_X\cdot \xi$, and so  $(K_X+B)\cdot \xi <0$.

Since $(X, B)$ is log canonical we may realise
our contraction as a $(K_X+B)$-negative contraction which exists by \cite[Theorem 0.2]{Ambro01}.
In the case of a flipping contraction, we may again realise the $(K_{\mathcal F}+B)$-flip 
as a $(K_X+B)$-flip, which exists by \cite{Birkar12} or \cite{HX11}.
\end{proof}

\subsection{Malgrange's theorem}

Let us recall the following version of Malgrange's theorem for singular varieties essentially due to \cite{CLN08}.

\begin{theorem}
\label{t_malgrange}
Let $(X,x)$ be a germ of a normal variety and let $\mathcal F$ be a co-rank one foliation on $X$.
Suppose that 
\begin{enumerate}
\item $\dim \Sing X \leq \dim X -3$;
\item $\dim \Sing \mathcal F \leq \dim X - 3$;
\item $N^{[*]}_{\mathcal F}$ is a line bundle; and 
\item $X$ has rational singularities.
\end{enumerate}

Then $\mathcal F$ admits a holomorphic first integral.  Moreover, there exist germs of analytic functions $f, g \in \mathcal O_{X, x}$ with $g(x) \neq 0$, such that $N^{[*]}_{\mathcal F}$
is generated by $gdf$.
\end{theorem}

\begin{proof}
This is essentially a rephrasing of \cite[Main Theorem]{CLN08}, with some minor modifications 
which we now explain.

Let \(n\coloneqq\dim X\). If \(n<3\), assumptions (1) and (2) imply that
\(X\) and \(\mathcal F\) are nonsingular near \(x\), and the claim
follows from Frobenius theorem. We may therefore assume
that \(n\geq3\).

Let $X^\circ = X \setminus \Sing X$.
We recall that {\it ibid} implies the conclusion of the Theorem under the following four hypotheses:
\begin{enumerate}
\item[(a)] $H^1(X^\circ, \mathcal O_{X^\circ}) = 0$;
\item[(b)] $X$ is $k$-regular for $k = 0, 1$ (cf. Definition \ref{def_kregular}); 
\item[(c)] $\mathcal F$ is defined by a holomorphic germ of a 1-form $\omega$ such that on $X^\circ$ 
$\dim \Sing \omega \leq \dim X -3$;
\item[(d)] $\dim \Sing X \leq \dim X - 3$.
\end{enumerate}

Hypothesis (d) of  \cite[Main Theorem]{CLN08} is our Hypothesis (1) above. 
Hypothesis (c) is an easy consequence of Hypotheses (2) and (3).
Hypothesis (a) follows from Hypotheses (1) and (4) above and the fact 
that rational singularities are Cohen-Macaulay.

It remains to consider Hypothesis (b).  We observe that $0$-regularity of $X$ is 
implied by the fact that $X$ is normal (cf. Remark \ref{r_kregular}).  The role of $1$-regularity in the proof of \cite[Main Theorem]{CLN08}
is only needed to guarantee that if $\mu\colon \tilde X\to X$ is a resolution of singularities of $X$ and $\eta$
is a $1$-form on $X^\circ$, then $\mu^*\eta$ extends to a holomorphic $1$-form on $\tilde X$ and that $\mu^*\eta$ 
vanishes when restricted to any component of $\mu^{-1}(x)$.

We now verify that this property holds in our setting.
Let $E$ denote the normalisation of $\mu^{-1}(x)$ and consider the diagram 
\begin{center}
\begin{tikzcd}
E \arrow[r, "i"] \arrow[d, "m"]& \tilde X  \arrow[d, "\mu"]\\
\{x\} \arrow[r, "j"] & X.
\end{tikzcd}
\end{center}

If $\eta$ is a $1$-form on $X^\circ$, then $\eta$ naturally gives a reflexive $1$-form on all of $X$.  
By \cite[Theorem 1.10]{KS21} we have a morphism $d_{refl}\mu\colon \mu^*\Omega^{[1]}_X \to \Omega^1_{\tilde X}$ and the 
image of $\eta$ under this morphism gives our required extension of $\mu^*\eta$ to a holomorphic $1$-form on all of $\tilde X$ (which we will continue to denote by $\mu^*\eta$).
To verify that $\mu^*\eta$ vanishes on each component of $E$ it suffices to show that $d_{refl} i\circ d_{refl}\mu$ is the zero morphism.  We observe that
by \cite[Theorem 1.10]{KS21} we have that $d_{refl} i\circ d_{refl}\mu = d_{refl}m \circ d_{refl}j$ and since $\{x\}$ is a point, $d_{refl}j$ is the zero morphism and we can conclude.

With these alterations, the proof of \cite[Main Theorem]{CLN08} goes through to prove our claim.
\end{proof}

\begin{lemma}
\label{lem_quasi_smooth_preserved}
Let $X$ be a normal projective variety and let $\mathcal F$ be a co-rank one foliation on $X$.  Suppose that 
\begin{itemize}
\item $X$ is $\mathbb Q$-factorial; 
\item $X$ has terminal singularities; 
\item $\mathcal F$ is quasi-smooth; and
\item $\mathcal F$ has canonical singularities.
\end{itemize}  
Let $f\colon X \dashrightarrow X'$ be a step of a $K_{\mathcal F}$-MMP and let $\mathcal F'$ be the foliation induced on $X'$. 

Then
\begin{enumerate}
\item $X'$ is $\mathbb Q$-factorial;
\item $X'$ has terminal singularities;
\item $\mathcal F'$ is quasi-smooth; and 
\item $\mathcal F'$ has canonical singularities. 
\end{enumerate}
\end{lemma}
\begin{proof}
Item (4) is an immediate consequence of the negativity lemma.  By the proof of Theorem \ref{t_contraction} we know that $f$ is also a step in a $K_{X}$-MMP, hence Items (1) and (2) likewise follow, cf. \cite{KM98}.  It remains to show that (3) holds.
We now show that the conditions of Theorem \ref{t_malgrange} are satisfied.

Since $X'$ has terminal singularities it follows that $\dim \Sing X' \leq \dim X' - 3$.
We next show that $\dim \Sing \mathcal F' \leq \dim X' - 3$. 
Suppose, for the sake of contradiction, that there exists a component $Z \subset \Sing \mathcal F'$ such that $Z$ is of codimension two in $X'$.
Since $\mathcal F$ is quasi-smooth this is only possible if $Z \subset \Exc f^{-1}$.  Since $\mathcal F$ has canonical singularities, the negativity lemma implies that $\mathcal F'$ has terminal singularities at the generic point of $Z$.  However, an easy calculation shows that if $Z \subset \Sing \mathcal F'$, then the blow-up of $X'$ along $Z$ extracts a divisor of discrepancy zero, a contradiction.

By Proposition \ref{p_bott_vanishing} we know that $N^{[*]}_{\mathcal F}$ is $f$-trivial, and so the basepoint free theorem implies that $f_*N^{[*]}_{\mathcal F}$ is a line bundle.  Note that $f_*N^{[*]}_{\mathcal F}$ is equal to $N^{[*]}_{\mathcal F'}$ away from a subset of codimension at least two, hence the two sheaves are equal, and so $N^{[*]}_{\mathcal F'}$ is locally free.

Finally, $X'$ has rational singularities because $X'$ has terminal singularities.

We now apply Theorem \ref{t_malgrange} to see that at every point $x \in X'$, the foliation  $\mathcal F'$ admits a local first integral and $N^{[*]}_{\mathcal F'}$ is generated by a $1$-form $df$.  
To see that $\mathcal F'$ is quasi-smooth it then suffices to note that $\dim \Sing (X', \mathcal F') \leq \dim X' - 3 = \rk \mathcal F' - 2$.
\end{proof}

\section{Running the MMP}

\begin{theorem}
\label{t_mmp_scaling}
Let $X$ be a projective variety and let $\mathcal F$ be a foliation on $X$ of co-rank one.
Suppose that 
\begin{enumerate}
\item $X$ is $\mathbb Q$-factorial; 
\item $X$ has terminal singularities; 
\item $\mathcal F$ is quasi-smooth;
\item $\mathcal F$ has canonical singularities.
\end{enumerate}

Then we may run the $K_{\mathcal F}$-MMP. 
Moreover, if either $K_{\mathcal F}$ is big or $K_{\mathcal F}$ is not pseudo-effective, then we may choose the MMP so that it terminates.
\end{theorem}

\begin{proof}
By Theorem \ref{t_contraction}, all the required divisorial
contractions and flips exist, and are also steps of a $K_X$-MMP.
By Lemma \ref{lem_quasi_smooth_preserved}, the hypotheses on $X$
and $\mathcal F$ are preserved by these steps.

We first suppose that $K_{\mathcal F}$ is big.
Choose an effective $\mathbb Q$-divisor
$D\sim_{\mathbb Q}K_{\mathcal F}$ and a sufficiently small rational
number $\epsilon>0$ such that $(X,\epsilon D)$ is klt.
Then $\epsilon D$ is big.

Run a $K_{\mathcal F}$-MMP, and denote its successive models by
$(X_i,\mathcal F_i)$ and the strict transform of $D$ by $D_i$.
Then $D_i\sim_{\mathbb Q}K_{\mathcal F_i}$. Proposition \ref{p_bott_vanishing} 
implies that 
if $R_i$ is the ray contracted at the $i$-th step, then 
$R_i$ is 
\(
(K_{X_i}+\epsilon D_i)\)-negative. 
Consequently, each step is also a step of the ordinary
$(K_X+\epsilon D)$-MMP.
Thus $(X_i,\epsilon D_i)$ remains klt, and
\cite[Theorem 1.1]{Hashizume26} implies termination.

Suppose now that $K_{\mathcal F}$ is not pseudo-effective.
Choose an effective ample $\mathbb Q$-divisor $A$ such that
$K_{\mathcal F}+A$ is nef, and run a $K_{\mathcal F}$-MMP with
scaling of $A$.
Let $\lambda_i$ be the scaling numbers and let $A_i$ be the
strict transform of $A$ on $X_i$.
Set
\[
\tau:=\inf\{t\geq0\mid
K_{\mathcal F}+tA\text{ is pseudo-effective}\}>0.
\]

We claim that $\lambda_i\geq\tau$ for every $i$.
Indeed, since the scaling numbers are non-increasing, every
earlier step is $(K_{\mathcal F}+\lambda_iA)$-nonpositive.
On a common resolution of $X\dashrightarrow X_i$, the pullback
of $K_{\mathcal F}+\lambda_iA$ is therefore the sum of the
pullback of the nef divisor
$K_{\mathcal F_i}+\lambda_iA_i$ and an effective exceptional
divisor. Hence $K_{\mathcal F}+\lambda_iA$ is pseudo-effective,
which proves the claim.

Choose a rational number $0<\epsilon<\tau$ sufficiently small so 
that $(X,\epsilon A)$ is klt.
For every ray $R_i$ selected by the scaling program, we have that $R_i$ is 
\(
(K_{X_i}+\epsilon A_i)\)-negative. Thus, each step of this MMP is also  a step of the ordinary
$(K_X+\epsilon A)$-MMP.
Since its initial boundary $\epsilon A$ is big,
\cite[Theorem 1.1]{Hashizume26} again implies termination.
\end{proof}

\begin{corollary}
\label{cor_main}
Let $X$ be a smooth projective variety and let $\mathcal F$ be a smooth foliation on $X$  of co-rank one.
Then we may run the $K_{\mathcal F}$-MMP. 

Moreover, if either $K_{\mathcal F}$ is big or $K_{\mathcal F}$ is not pseudo-effective,
then we may run a $K_{\mathcal F}$-MMP $\phi\colon X\dashrightarrow X'$ such that, writing  $\mathcal F' := \phi_*\mathcal F$, the following hold: 
\begin{enumerate}
\item if $K_{\mathcal F}$ is big, then $\mathcal F'$ is a minimal model, i.e. $K_{\mathcal F'}$ is nef; or
\item if $K_{\mathcal F}$ is not pseudo-effective, then $\mathcal F'$ admits a Mori fibre space, i.e. a contraction  
$\eta\colon X'\to Z$ onto a normal projective variety $Z$ such that $\dim Z<\dim X$, 
$\rho(X'/Z)=1$ and $-K_{\mathcal F'}$ is $\eta$-ample.  
\end{enumerate}

\end{corollary}
\begin{proof}
Since $X$ and $\mathcal F$ are smooth, $\mathcal F$ is quasi-smooth.
We may apply Theorem \ref{t_mmp_scaling} to conclude.
\end{proof}

\bibliography{math.bib}
\bibliographystyle{alpha}

\end{document}